\documentclass[aop,MSNbibl,seceqn,noautosecdot,dvips]{arximspdf}
\usepackage{mathrsfs,url,breakurl}

\doi{10.1214/13-AOP892} 
\volume{42}
\issue{3}
\pubyear{2014}
\firstpage{906}
\lastpage{945}

\makeatletter
\newcommand{\vertt}{\mid}
\newcommand{\mathrell}{:}

\newcommand{\rright}{\right}
\newcommand{\lleft}{\left}
\newcommand{\rrVert}{\Vert}
\newcommand{\rrvert}{\vert}
\newcommand{\llVert}{\Vert}
\newcommand{\llvert}{\vert}
\newcommand{\eqref}[1]{(\ref{#1})}
\newcommand{\metricspace}{\mathcal{Z}}

\newtheorem{thmm}{Theorem}[section]
\newtheorem{cor}[thmm]{Corollary}
\newtheorem{lemma}[thmm]{Lemma}
\newtheorem{prop}[thmm]{Proposition}

\newproclaim{assumption}[thmm]{Assumption}

\newproclaim{defn}[thmm]{Definition}
\newproclaim{example}[thmm]{Example}

\newproclaim{rem}[thmm]{Remark}

\newcommand{\cnst}[1]{\mathrm{#1}}

\newcommand{\eps}{\varepsilon}

\newcommand{\econst}{\mathrm{e}}

\newcommand{\Id}{\mathbf{I}}
\newcommand{\zeromtx}{\mathbf{0}}

\newcommand{\coll}[1]{\mathscr{#1}}

\newcommand{\R}{\mathbb{R}}
\newcommand{\C}{\mathbb{C}}

\newcommand{\M}{\mathbb{M}}

\newcommand{\sgn}[1]{\operatorname{sgn}{#1}}

\newcommand{\Expect}{\operatorname{\mathbb{E}}}

\newcommand{\ent}{\operatorname{ent}}

\newcommand{\vct}[1]{\bolds{#1}}
\newcommand{\vctt}[1]{\mathbf{#1}}
\newcommand{\mtx}[1]{\mathbf{#1}}
\newcommand{\mtxx}[1]{\bolds{#1}}

\newcommand{\diag}{\operatorname{diag}}
\newcommand{\trace}{\operatorname{tr}}
\newcommand{\ntr}{\operatorname{\bar{tr}}}

\newcommand{\psdle}{\preccurlyeq}
\newcommand{\psdge}{\succcurlyeq}

\makeatother

\begin{document}
\begin{frontmatter}

\title{Matrix concentration inequalities via the method of exchangeable
pairs\thanksref{T1}}
\runtitle{Matrix concentration via exchangeable pairs}

\thankstext{T1}{Supported in part by
the U.S. Army Research Laboratory and the U.S. Army Research
Office under Contract/Grant number W911NF-11-1-0391.}

\begin{aug}
\author[a]{\fnms{Lester} \snm{Mackey}\corref{}\ead[label=e1]{lmackey@stanford.edu}\thanksref{t2}},
\author[b]{\fnms{Michael I.} \snm{Jordan}\ead[label=e2]{jordan@stat.berkeley.edu}},
\author[c]{\fnms{Richard Y.} \snm{Chen}\ead[label=e3]{ycchen@caltech.edu}\thanksref{t3}},
\author[c]{\fnms{Brendan}~\snm{Farrell}\ead[label=e4]{farrell@cms.caltech.edu}\thanksref{t3}}
\and
\author[c]{\fnms{Joel A.} \snm{Tropp}\ead[label=e5]{jtropp@cms.caltech.edu}\thanksref{t3}}

\thankstext{t2}{Supported by the National Defense Science and Engineering
Graduate Fellowship.}
\thankstext{t3}{Supported by
ONR awards N00014-08-1-0883 and N00014-11-1002,
AFOSR award FA9550-09-1-0643,
DARPA award N66001-08-1-2065
and a Sloan Research Fellowship.}
\runauthor{L. Mackey et al.}

\affiliation{Stanford University,
University of California, Berkeley,
California~Institute of Technology,
California~Institute of Technology
and~California Institute of Technology}

\address[a]{L. Mackey\\
Department of Statistics\\
Stanford University\\
Sequoia Hall\\
390 Serra Mall\\
Stanford, California 94305-4065\\
USA\\
\printead{e1}}

\address[b]{M. I. Jordan\\
Departments of EECS and Statistics\\
University of California, Berkeley\\
427 Evans Hall\\
Berkeley, California 94720\\
USA\\
\printead{e2}}

\address[c]{R.~Y. Chen\\
B. Farrell\\
J.~A. Tropp\\
Department of Computing\\
\quad and Mathematical Sciences\\
California Institute of Technology\\
1200 E. California Blvd.\\
Pasadena, California 91125\\
USA\\
\printead{e3} \\
\phantom{E-mail:\ }\printead*{e4}\\
\phantom{E-mail:\ }\printead*{e5}}
\end{aug}
%

\received{\smonth{2} \syear{2012}}
\revised{\smonth{2} \syear{2013}}

%
\begin{abstract}
This paper derives exponential concentration inequalities and
polynomial moment inequalities for the spectral norm of a random
matrix. The analysis requires a matrix extension of the scalar
concentration theory developed by Sourav Chatterjee using Stein's
method of exchangeable pairs. When applied to a sum of independent
random matrices, this approach yields matrix generalizations of the
classical inequalities due to Hoeffding, Bernstein, Khintchine and
Rosenthal. The same technique delivers bounds for sums of dependent
random matrices and more general matrix-valued functions of dependent
random variables.
\end{abstract}

%

\end{frontmatter}

\section{\texorpdfstring{Introduction.}{Introduction}}

Matrix concentration inequalities control the fluctuations of a random
matrix about its mean.
At present, these results provide an effective method for studying sums of
independent random matrices and matrix martingales \cite
{Oli10Concentration-Adjacency,Tro11Freedmans-Inequality,Tro11User-Friendly-FOCM,Min11Some-Extensions}.
They have been used to streamline
the analysis of structured random matrices
in a range of applications, including
statistical estimation \cite{Kol11Oracle-Inequalities}, randomized
linear algebra \cite{Git11Spectral-Norm,CD11Sublinear-Randomized},
stability of least-squares approximation \cite
{CDL11Stability-Accuracy}, combinatorial and robust optimization \cite
{So11,CheungSoWa11}, matrix completion \mbox{\cite
{Gro11Recovering-Low-Rank,Rec11Simpler-Approach,NegahbanWa12,MackeyTaJo11}}
and random graph theory \cite{Oli10Concentration-Adjacency}.
These works compose only a small sample of the papers that rely on matrix
concentration inequalities. Nevertheless, %
it remains common to encounter new classes of random matrices that we
cannot treat with
the available techniques.

The purpose of this paper is to lay the foundations of a new approach
for analyzing structured random matrices. Our work is based on
Chatterjee's technique for developing scalar concentration
inequalities \cite
{Cha07Steins-Method,Cha08Concentration-Inequalities} via Stein's
method of exchangeable pairs \cite{Stein72}. We extend this argument to
the matrix setting, where we use it to establish exponential
concentration bounds (Theorems \ref{thmm:concentration-bdd} and \ref
{thmm:concentration-subgauss}) and polynomial moment inequalities
(Theorem~\ref{thmm:BDG-inequality}) for the spectral norm of a random matrix.

To illustrate the power of this idea, we show that our general results
imply several important
concentration bounds for a sum of independent, random, Hermitian
matrices \cite
{LPP91Noncommutative-Khintchine,JX03Noncommutative-Burkholder,Tro11User-Friendly-FOCM}.
In particular, we obtain a matrix Hoeffding inequality with optimal
constants (Corollary~\ref{cor:hoeffding}) and a version of the matrix
Bernstein inequality (Corollary~\ref{cor:bernstein}). Our techniques
also yield concise proofs of the matrix Khintchine inequality
(Corollary~\ref{cor:khintchine}) and the matrix Rosenthal inequality
(Corollary~\ref{cor:ros-pin}).

The method of exchangeable pairs also applies to matrices constructed
from dependent random variables.
We offer a hint of the prospects by establishing concentration results
for several other classes of random matrices.
In Section~\ref{sec:cond-zero-mean}, we consider sums of dependent
matrices that satisfy a conditional zero-mean property.
In Section~\ref{sec:combinatorial-sum}, we treat a broad class of
combinatorial matrix statistics.
Finally, in Section~\ref{sec:self-repro}, we analyze general
matrix-valued functions that have a self-reproducing property.
%

%

%

%

%

%

%

%

%

\subsection{\texorpdfstring{Notation and preliminaries.}{Notation and preliminaries}}

The symbol $\llVert {\cdot} \rrVert $ is reserved for the
spectral norm, which
returns the largest singular value
of a general complex matrix.

We write $\M^d$ for the algebra of all $d \times d$ complex matrices.
The trace and normalized trace
of a square matrix are defined as
\[
\trace\mtx{B} \mathrell=\sum_{j=1}^d
b_{jj}\quad \mbox{and} \quad\ntr\mtx{B} \mathrell=\frac{1}{d} \sum
_{j=1}^d b_{jj} \qquad\mbox{for $
\mtx{B} \in\M^{d}$.}
\]

We define the linear space $\mathbb{H}^{d}$ of Hermitian $d \times d$ matrices.
\emph{All matrices in this paper are Hermitian unless explicitly stated
otherwise.}
The symbols $\lambda_{\max}(\mtx{A})$ and $\lambda_{\min}(\mtx{A})$
refer to the algebraic
maximum and minimum eigenvalues of a matrix $\mtx{A} \in\mathbb{H}^{d}$.
For each interval $I \subset\R$, we define the set of Hermitian matrices
whose eigenvalues fall in that interval,
\[
\mathbb{H}^{d}(I) \mathrell=\bigl\{ \mtx{A} \in
\mathbb{H}^{d} \dvtx\bigl[\lambda_{\min}(\mtx{A}),
\lambda_{\max}(\mtx{A})\bigr] \subset I \bigr\}.
\]
The set $\mathbb{H}_{+}^{d}$ consists of all
positive-semidefinite (psd) $d \times d$ matrices.
Curly inequalities refer to the semidefinite partial order on Hermitian
matrices.
For example, we write $\mtx{A} \psdle\mtx{B}$ to signify that the matrix
$\mtx{B} - \mtx{A}$ is psd.

We require operator convexity properties of the matrix square so often
that we state them now:
%
\begin{equation}
\label{eqn:square-convex} \biggl(\frac{\mtx{A} + \mtx{B}}{2} \biggr)^2 \psdle
\frac{\mtx
{A}^2 +
\mtx{B}^2}{2} \qquad\mbox{for all $\mtx{A}, \mtx{B} \in\mathbb{H}^{d}$.}
\end{equation}
More generally, we have the operator Jensen inequality
%
\begin{equation}
\label{eqn:kadison} (\Expect\mtx{X})^2 \psdle\Expect\mtx{X}^2,
\end{equation}
valid for any random Hermitian matrix, provided that $\Expect\llVert {\mtx{X}} \rrVert ^2 < \infty$. To verify this result,
simply expand the inequality $\Expect(\mtx{X} - \Expect\mtx{X})^2
\psdge\mtx{0}$.
The operator Jensen inequality also holds for conditional expectation,
again provided that $\Expect\llVert {\mtx{X}} \rrVert ^2 <
\infty$.
%

%

%

%

\section{\texorpdfstring{Exchangeable pairs of random matrices.}
{Exchangeable pairs of random matrices}} \label{sec:exchange}

Our approach to studying random matrices is based on the method of
exchangeable pairs, which
originates in the work of Charles Stein \cite{Stein72} on normal
approximation for a sum of
dependent random variables. In this section, we explain how some
central ideas from this theory
extend to matrices.

\subsection{\texorpdfstring{Matrix Stein pairs.}{Matrix Stein pairs}} \label{sec:pairs}

First, we define an exchangeable pair.

\begin{defn}[(Exchangeable pair)] \label{def:exchange}
Let $Z$ and $Z'$ be random variables taking values in a Polish space %
$\metricspace$.
We say that $(Z, Z')$ is an \emph{exchangeable pair} if it has the same
distribution as $(Z', Z)$.
In particular, $Z$ and $Z'$ must share the same distribution.
\end{defn}

We can obtain a lot of information about the fluctuations of a random
matrix~$\mtx{X}$
if we can construct a good exchangeable pair $(\mtx{X}, \mtx{X}')$.
With this motivation in mind, let us introduce a special class of
exchangeable pairs.
%

%

%

\begin{defn}[(Matrix Stein pair)] \label{def:stein-pair}
Let $(Z, Z')$ be an exchangeable pair of random variables taking values
in a Polish space $\metricspace$,
and let $\mtxx{\Psi} \dvtx\metricspace\to\mathbb{H}^{d}$ be a
measurable function.
Define the random Hermitian matrices
\[
\mtx{X} \mathrell=\mtxx{\Psi}(Z)\quad \mbox{and}\quad \mtx{X}'
\mathrell=\mtxx{\Psi}\bigl(Z'\bigr).
\]
We say that $(\mtx{X}, \mtx{X}')$ is a \emph{matrix Stein pair} if
there is a constant $\alpha\in(0, 1]$
for which
%
\begin{equation}
\label{eqn:F-and-G} \Expect\bigl[ \mtx{X} - \mtx{X}' \vertt Z \bigr] =
\alpha\mtx{X}\qquad \mbox{almost surely.}
\end{equation}
The constant $\alpha$ is called the \emph{scale factor} of the pair.
When discussing a matrix Stein pair $(\mtx{X}, \mtx{X}')$, we always
assume that $\Expect\llVert {\mtx{X}} \rrVert ^2 < \infty$.
\end{defn}

A matrix Stein pair $(\mtx{X}, \mtx{X}')$ has several
useful properties. First, $(\mtx{X}, \mtx{X}')$ always forms an
exchangeable pair.
Second, it must be the case that $\Expect\mtx{X} = \mtx{0}$. Indeed,
\[
\Expect\mtx{X} = \frac{1}{\alpha}\Expect \bigl[ \Expect\bigl[ \mtx{X} - \mtx
{X}' \vertt Z \bigr] \bigr] = \frac{1}{\alpha}\Expect\bigl[ \mtx{X}
- \mtx{X}' \bigr] = \mtx{0}
\]
because of identity \eqref{eqn:F-and-G}, the tower property of
conditional expectation and the exchangeability of $(\mtx{X}, \mtx
{X}')$. In Section~\ref{sec:indep-sum}, we construct a matrix Stein
pair for a sum of centered, independent random matrices. More
sophisticated examples appear in Sections~\ref{sec:cond-zero-mean}, \ref
{sec:combinatorial-sum} and \ref{sec:self-repro}.

\begin{rem}[(Approximate matrix Stein pairs)]
In the scalar setting, it is common to consider exchangeable pairs that
satisfy an approximate Stein condition. For matrices, this condition
reads $\Expect[ \mtx{X} - \mtx{X}' \vertt Z ] = \alpha\mtx{X} +
\mtx
{R}$, where $\mtx{R}$ is an error term. The methods in this paper
extend easily to this case.
\end{rem}

%

%

%

%

\subsection{\texorpdfstring{The method of exchangeable pairs.}{The method of exchangeable pairs}} \label{sec:method-exchange}

A well-chosen matrix Stein pair $(\mtx{X}, \mtx{X}')$ provides a surprisingly
powerful tool for studying the random matrix $\mtx{X}$. The technique
depends on a
fundamental technical lemma.

%

\begin{lemma}[(Method of exchangeable pairs)] \label{lem:exchange}
Suppose that $(\mtx{X}, \mtx{X}') \in\mathbb{H}^{d} \times\mathbb
{H}^{d}$ is a
matrix Stein pair with scale factor $\alpha$.
Let $\mtx{F}\dvtx\mathbb{H}^{d} \rightarrow\mathbb{H}^{d}$ be a
measurable function
that satisfies the regularity condition
%
\begin{equation}
\label{eqn:regularity-mep} \Expect\bigl\llVert {\bigl(\mtx{X} - \mtx{X}'\bigr)
\cdot\mtx{F}(\mtx{X}) } \bigr\rrVert < \infty.
\end{equation}
Then
%
\begin{equation}
\label{eqn:exchange} \Expect \bigl[ \mtx{X} \cdot\mtx{F}(\mtx{X}) \bigr] =
\frac{1}{2\alpha} \Expect \bigl[ \bigl(\mtx{X} - \mtx{X}'\bigr)
\bigl(\mtx {F}(\mtx {X}) - \mtx{F}\bigl(\mtx{X}'\bigr) \bigr) \bigr].
\end{equation}
\end{lemma}

In short, the %
randomness in the Stein pair furnishes an alternative expression
for the expected product of $\mtx{X}$ and the %
function $\mtx{F}$. Identity \eqref{eqn:exchange} is valuable
because it allows us to estimate this integral using the smoothness
properties of the function~$\mtx{F}$ and the discrepancy between $\mtx{X}$ and $\mtx{X}'$.

\begin{pf*}{Proof of Lemma \ref{lem:exchange}}
Suppose $(\mtx{X}, \mtx{X}')$ is a matrix Stein pair constructed from
an auxiliary exchangeable pair $(Z, Z')$.
The defining property \eqref{eqn:F-and-G} implies
\[
\alpha\cdot\Expect\bigl[ \mtx{X} \cdot\mtx{F}(\mtx{X}) \bigr] = \Expect \bigl[
\Expect\bigl[ \mtx{X} - \mtx{X}' \vertt Z \bigr] \cdot\mtx {F}(
\mtx{X}) \bigr] = \Expect\bigl[ \bigl(\mtx{X} - \mtx{X}'\bigr)
\mtx{F}(\mtx{X}) \bigr].
\]
We have used regularity condition \eqref{eqn:regularity-mep} to invoke
the pull-through property of conditional expectation. Since $(\mtx{X},
\mtx{X}')$ is an exchangeable pair,
\[
\Expect\bigl[ \bigl(\mtx{X} - \mtx{X}'\bigr) \mtx{F}(
\mtx{X}) \bigr] = \Expect\bigl[ \bigl(\mtx{X}' - \mtx{X}\bigr)
\mtx{F}\bigl(\mtx{X}'\bigr) \bigr] = - \Expect\bigl[ \bigl(\mtx{X} -
\mtx{X}'\bigr) \mtx{F}\bigl(\mtx{X}'\bigr) \bigr].
\]
Identity \eqref{eqn:exchange} follows when we average the two preceding
displays.
\end{pf*}

\subsection{\texorpdfstring{The conditional variance.}{The conditional variance}}

To each matrix Stein pair $(\mtx{X}, \mtx{X}')$, we may associate a
random matrix called the \emph{conditional variance} of $\mtx{X}$. The
ultimate purpose of this paper is to argue that the spectral norm of
$\mtx{X}$
is unlikely to be large when the conditional variance is small.

\begin{defn}[(Conditional variance)]
Suppose that $(\mtx{X}, \mtx{X}')$ is a matrix Stein pair,
constructed from an auxiliary exchangeable pair $(Z, Z')$.
The \emph{conditional variance} is the random matrix
%
\begin{equation}
\label{eqn:conditional-variance} %
\mtxx{\Delta}_{\mtx{X}} \mathrell=\mtxx{
\Delta}_{\mtx{X}}(Z) \mathrell=\frac{1}{2 \alpha} \Expect \bigl[
\bigl(\mtx{X} - \mtx{X}'\bigr)^2 \vertt Z \bigr],
\end{equation}
where $\alpha$ is the scale factor of the pair.
We may take any version of the conditional expectation in this definition.
\end{defn}

The conditional variance $\mtxx{\Delta}_{\mtx{X}}$ should be regarded as
a stochastic estimate for the variance of the random matrix $\mtx{X}$. Indeed,
%
\begin{equation}
\label{eqn:mean-delta} \Expect[ \mtxx{\Delta}_{\mtx{X}} ] = \Expect
\mtx{X}^2.
\end{equation}
This identity follows %
from Lemma~\ref{lem:exchange} with the choice $\mtx{F}(\mtx{X}) =
\mtx{X}$.
%

%

\subsection{\texorpdfstring{Example: A sum of independent random matrices.}
{Example: A sum of independent random matrices}} \label
{sec:indep-sum}

To make the definitions in this section more vivid, we describe a
simple but important example of a matrix Stein
pair. Consider an independent sequence $Z \mathrell=( \mtx
{Y}_1, \ldots,
\mtx{Y}_n )$ of random Hermitian matrices that
satisfies $\Expect\mtx{Y}_k = \mtx{0}$ and $\Expect\llVert
{\mtx{Y}_k} \rrVert ^2
< \infty$ for each $k$.
Introduce the random series
\[
\mtx{X} \mathrell=\mtx{Y}_1 + \cdots+ \mtx{Y}_n.
\]

Let us explain how to build a good matrix Stein pair $(\mtx{X}, \mtx
{X}')$. We need the exchangeable counterpart $\mtx{X}'$ to have the
same distribution as $\mtx{X}$, but it should also be close to $\mtx
{X}$ so that we can control the conditional variance. To achieve these
goals, we construct $\mtx{X}'$ by picking a summand from $\mtx{X}$ at
random and replacing it with a fresh copy.

Formally, let $\mtx{Y}_k'$ be an independent copy of $\mtx{Y}_k$ for
each index $k$,
and draw a random index $K$ uniformly from $\{ 1, \ldots, n \}$ and
independently
from everything else.
Define the random sequence
\[
Z' \mathrell=\bigl(\mtx{Y}_1, \ldots,
\mtx{Y}_{K-1}, \mtx {Y}_K',
\mtx{Y}_{K+1}, \ldots, \mtx{Y}_n\bigr).
\]
One can check that $(Z, Z')$
forms an exchangeable pair. The random matrix
\[
\mtx{X}' \mathrell=\mtx{Y}_1 + \cdots+
\mtx{Y}_{K-1} + \mtx{Y}_K' + \mtx
{Y}_{K+1} + \cdots+ \mtx{Y}_n %
\]
is thus an exchangeable counterpart for $\mtx{X}$.
To verify that $(\mtx{X}, \mtx{X}')$ is a Stein pair, calculate that
\begin{eqnarray*}
\Expect\bigl[\mtx{X} - \mtx{X}' \vertt Z \bigr] &=& \Expect\bigl[
\mtx{Y}_K - \mtx{Y}_K' \vertt Z \bigr]
\\
&=& \frac{1}{n} \sum_{k=1}^n \Expect
\bigl[ \mtx{Y}_k - \mtx{Y}_k' \vertt Z
\bigr] = \frac{1}{n} \sum_{k=1}^n
\mtx{Y}_k = \frac{1}{n} \mtx{X}.
\end{eqnarray*}
The third identity holds because $\mtx{Y}_k'$ is a centered random
matrix that is independent from $Z$. Therefore, $(\mtx{X}, \mtx{X}')$
is a matrix Stein pair with scale factor $\alpha= n^{-1}$.

Next, we compute the conditional variance:
%
\begin{eqnarray}
\label{eqn:indep-sum-DeltaX} \mtxx{\Delta}_{\mtx{X}} &=& \frac{n}{2} \cdot\Expect
\bigl[ \bigl(\mtx{X} - \mtx{X}'\bigr)^2 \vertt Z \bigr]
\nonumber
\\
&=& \frac{n}{2} \cdot\frac{1}{n} \sum
_{k=1}^n \Expect \bigl[ \bigl(\mtx{Y}_k
- \mtx{Y}_k'\bigr)^2 \vertt Z \bigr]
\nonumber
\\[-8pt]
\\[-8pt]
\nonumber
&=& \frac{1}{2} \sum_{k=1}^n \bigl[
\mtx{Y}_k^2 - \mtx{Y}_k \bigl(\Expect
\mtx{Y}_k'\bigr) - \bigl(\Expect\mtx{Y}_k'
\bigr) \mtx{Y}_k + \Expect\bigl(\mtx{Y}_k'
\bigr)^2 \bigr]
\\
&=& \frac{1}{2} \sum_{k=1}^n \bigl(
\mtx{Y}_k^2 + \Expect \mtx {Y}_k^2
\bigr).\nonumber
\end{eqnarray}
For the third relation, expand the square and invoke the pull-through
property of conditional expectation. We may drop the conditioning
because $\mtx{Y}_k'$ is independent from $Z$. In the last line, we
apply the property that $\mtx{Y}_k'$ has the same distribution as~$\mtx{Y}_k$.

Expression \eqref{eqn:indep-sum-DeltaX} shows that we can control the
size of the conditional expectation uniformly if we can control the
size of the individual summands. This example also teaches us that we
may use the symmetries of the distribution of the random matrix to
construct a matrix Stein pair.

%

%

%

%

%

\section{\texorpdfstring{Exponential moments and eigenvalues of a random matrix.}
{Exponential moments and eigenvalues of a random matrix}}\label
{sec:trace-moments}

Our main goal in this paper is to study the behavior of the extreme
eigenvalues of a random
Hermitian matrix. In Section~\ref{sec:matrix-laplace}, we describe an
approach to this
problem that parallels the classical Laplace transform method
for scalar random variables. The adaptation to the matrix setting leads
us to consider
the \emph{trace} of the moment generating function (m.g.f.) of a random matrix.
After presenting this background, we explain how the method of
exchangeable pairs can be used
to control the growth of the trace m.g.f. This result, which appears in
Section~\ref{sec:d-trace-mgf},
is the key to our exponential concentration bounds for random matrices.

%

%

%

\subsection{\texorpdfstring{Standard matrix functions.}{Standard matrix functions}} \label{sec:matrix-function}

Before entering the discussion, recall that a \emph{standard matrix function}
is obtained by applying a real function to the eigenvalues of a
Hermitian matrix.
Higham \cite{Hig08Functions-Matrices} provides an excellent treatment
of this
concept.

%

%
\begin{defn}[(Standard matrix function)] \label{def:std-fn}
Let $f \dvtx I \to\R$ be a function on an interval $I$ of the real line.
Suppose that $\mtx{A} \in\mathbb{H}^{d}(I)$ has the eigenvalue
decomposition $\mtx{A} = \mtx{Q} \cdot\diag(\lambda_1, \ldots,
\lambda
_d) \cdot\mtx{Q}^*$ where $\mtx{Q}$ is a unitary matrix. Then
the matrix extension
$
f(\mtx{A}) \mathrell=\mtx{Q} \cdot\diag(f(\lambda_1),
\ldots,
f(\lambda
_d)) \cdot\mtx{Q}^*.
$
\end{defn}

The \emph{spectral mapping theorem} states that, if $\lambda$ is an
eigenvalue of $\mtx{A}$, then $f(\lambda)$ is an eigenvalue of
$f(\mtx{A})$.
This fact follows %
from Definition~\ref{def:std-fn}.

When we apply a familiar scalar function to a Hermitian matrix, we are
always referring to a standard matrix function. For instance,
$\llvert {\mtx {A}} \rrvert $ is the matrix absolute value, $\exp
(\mtx{A})$ is the matrix
exponential, and $\log(\mtx{A})$ is the matrix logarithm. The latter is
defined only for positive-definite matrices.
%

\subsection{\texorpdfstring{The matrix Laplace transform method.}
{The matrix Laplace transform method}} \label{sec:matrix-laplace}

Let us introduce a matrix variant of the classical moment generating function.
We learned this definition from Ahlswede--Winter \cite{AW02Strong-Converse},
Appendix.
%

%

%

\begin{defn}[(Trace m.g.f.)]
Let $\mtx{X}$ be a random Hermitian matrix. The
\emph{(normalized) trace moment generating function} of $\mtx{X}$ is
defined as
\[
m(\theta) \mathrell=m_{\mtx{X}}(\theta) \mathrell=
\Expect\ntr\econst ^{\theta\mtx{X}} \qquad\mbox{for $\theta\in\R$.}
\]
We admit the possibility that the expectation may not exist for all %
$\theta$.
\end{defn}

Ahlswede and Winter \cite{AW02Strong-Converse}, Appendix, had the insight that
the classical Laplace transform method could be extended to the matrix
setting by
replacing the classical m.g.f. with the trace m.g.f. This adaptation allows
us to
obtain concentration inequalities for the extreme eigenvalues of a random
Hermitian matrix using methods from matrix analysis. %
The following proposition distills results from the
papers \cite
{AW02Strong-Converse,Oli10Sums-Random,Tro11User-Friendly-FOCM,CGT11Masked-Sample}.
%

\begin{prop} [(Matrix Laplace transform method)] \label{prop:matrix-laplace}
Let $\mtx{X} \in\mathbb{H}^{d}$ be a random matrix with %
trace m.g.f.
$m(\theta) \mathrell=\Expect\ntr\econst^{\theta\mtx
{X}}$. For each $t
\in\R$,
%
\begin{eqnarray}
\mathbb{P} \bigl\{ {\lambda_{\max} ( \mtx{X}) \geq t} \bigr\} &\leq& d
\cdot\inf_{\theta> 0}  \exp\bigl\{ -\theta t + \log m(\theta) \bigr
\}, \label{eqn:laplace-upper-tail}
\\
\mathbb{P} \bigl\{ {\lambda_{\min}(\mtx{X}) \leq t} \bigr\} &\leq& d
\cdot\inf_{\theta< 0}  \exp\bigl\{ -\theta t + \log m(\theta) \bigr
\}. \label{eqn:laplace-lower-tail}
\end{eqnarray}
Furthermore,
%
\begin{eqnarray}
\Expect\lambda_{\max}(\mtx{X})& \leq&\inf_{\theta> 0}
\frac{1}{\theta} \bigl[\log d + \log m(\theta)\bigr], \label{eqn:laplace-upper-mean}
\\
\Expect\lambda_{\min}(\mtx{X})& \geq&\sup_{\theta< 0}  \frac{1}{\theta} \bigl[\log d + \log m(\theta)\bigr]. \label{eqn:laplace-lower-mean}
\end{eqnarray}
\end{prop}

Estimates \eqref{eqn:laplace-upper-mean} and \eqref
{eqn:laplace-lower-mean} for the expectations are usually sharp up to
the logarithm of the dimension. In many situations, tail bounds \eqref
{eqn:laplace-upper-tail} and \eqref{eqn:laplace-lower-tail} are
reasonable for moderate $t$, but they tend to overestimate the
probability of a large deviation.
Note that, in general, we cannot dispense with the dimensional factor $d$.
See \cite{Tro11User-Friendly-FOCM}, Section~4, for a detailed
discussion of these issues.
Additional inequalities for the interior eigenvalues can be established
using the
minimax Laplace transform method \cite{GittensTr11}.

\begin{pf*}{Proof of Proposition \ref{prop:matrix-laplace}}
To establish \eqref{eqn:laplace-upper-tail}, fix $\theta> 0$. Owing to
Markov's inequality,
\begin{eqnarray*}
\mathbb{P} \bigl\{ {\lambda_{\max} (\mtx{X}) \geq t} \bigr\} &=&
\mathbb{P} \bigl\{ {\econst^{\lambda_{\max}(\theta\mtx{X})} \geq\econst ^{\theta t}} \bigr\}
\leq\econst^{-\theta t} \cdot\Expect\econst^{\lambda_{\max}
(\theta
\mtx{X})}
\\
&= &\econst^{-\theta t} \cdot\Expect\lambda_{\max} \bigl(\econst
^{\theta\mtx{X}} \bigr) \leq\econst^{-\theta t} \cdot\Expect\trace
\econst^{\theta\mtx{X}}.
\end{eqnarray*}
The third relation depends on the spectral mapping theorem and the
monotonicity of the exponential. The last inequality holds because the
trace of a positive-definite matrix exceeds its maximum eigenvalue.
Identify the normalized trace m.g.f., and take the infimum over $\theta$
to complete the argument.

The proof of \eqref{eqn:laplace-lower-tail} parallels the proof
of \eqref{eqn:laplace-upper-tail}.
For $\theta< 0$,
\[
\mathbb{P} \bigl\{ {\lambda_{\min}(\mtx{X}) \leq t} \bigr\} = \mathbb{P}
\bigl\{ {\theta\lambda_{\min}(\mtx{X}) \geq\theta t} \bigr\} = \mathbb{P}
\bigl\{ {\lambda_{\max}(\theta\mtx{X}) \geq\theta t} \bigr\}.
\]
We used the property that $-\lambda_{\min}(\mtx{A}) = \lambda_{\max
}(-\mtx{A})$
for each Hermitian matrix $\mtx{A}$. The rest of the argument is the
same as in the preceding paragraph.

For the expectation bound \eqref{eqn:laplace-upper-mean}, fix $\theta>
0$. Jensen's inequality yields
\[
\Expect\lambda_{\max}(\mtx{X}) = \theta^{-1} \Expect
\lambda_{\max}(\theta\mtx{X}) \leq\theta^{-1} \log\Expect
\econst^{\lambda_{\max}(\theta\mtx{X})} \leq\theta^{-1} \log\Expect\trace
\econst^{\theta\mtx{X}}.
\]
The justification is the same as above. Identify the normalized trace
m.g.f., and take the infimum over $\theta> 0$.
Similar considerations yield \eqref{eqn:laplace-lower-mean}.
\end{pf*}

%

\subsection{\texorpdfstring{Studying the trace m.g.f. with exchangeable pairs.}
{Studying the trace m.g.f. with exchangeable pairs}}

The technical difficulty in the matrix Laplace transform method
arises because we need to estimate the trace m.g.f. Previous
authors have applied deep results from matrix analysis to accomplish
this bound:
the Golden--Thompson inequality is central to \cite
{AW02Strong-Converse,Oli10Concentration-Adjacency,Oli10Sums-Random},
while Lieb's result \cite{Lie73Convex-Trace}, Theorem~6,
animates \cite
{Tro11Freedmans-Inequality,Tro11User-Friendly-FOCM,HKZ12Dimension-Free-Tail}.

In this paper, we develop a fundamentally different technique for
studying the trace m.g.f.
The main idea is to control the \emph{growth} of the trace m.g.f. by bounding
its \emph{derivative}.
To see why we have adopted this strategy, consider a random
Hermitian matrix $\mtx{X}$, and observe that the
derivative of its trace m.g.f. can be written as
\[
m'(\theta) = \Expect\ntr \bigl[ \mtx{X} \econst^{\theta\mtx{X}}
\bigr]
\]
under appropriate regularity conditions. This expression has just the
form that we need to invoke the method of exchangeable pairs,
Lemma~\ref
{lem:exchange},
with $\mtx{F}(\mtx{X}) = \econst^{\theta\mtx{X}}$. We obtain
%
\begin{equation}
\label{eqn:m-prime-temp} m'(\theta) = \frac{1}{2 \alpha} \Expect\ntr \bigl[
\bigl(\mtx{X} - \mtx{X}'\bigr) \bigl( \econst^{\theta\mtx{X}} -
\econst^{\theta\mtx{X}'} \bigr) \bigr].
\end{equation}
This formula strongly suggests that we should apply a mean value
theorem to
control the derivative; we establish the result that we need in
Section~\ref{sec:mvti} below.
Ultimately, this argument leads to a differential inequality for
$m'(\theta)$,
which we can integrate to obtain an estimate for $m(\theta)$.

The technique of bounding the derivative of an m.g.f. lies at the heart
of the log-Sobolev method for studying concentration
phenomena \cite{Led01Concentration-Measure}, Chapter~5.
Recently, Chatterjee \cite
{Cha07Steins-Method,Cha08Concentration-Inequalities} demonstrated
that the method of exchangeable pairs provides another way to control
the derivative of
an m.g.f. Our arguments closely follow the pattern set by Chatterjee;
the novelty inheres in the
extension of these ideas to the matrix setting
and the striking applications that this extension permits.

%

%

%

\subsection{\texorpdfstring{The mean value trace inequality.}{The mean value trace inequality}} \label{sec:mvti}

To bound expression \eqref{eqn:m-prime-temp} for the derivative of the
trace m.g.f., we need a matrix generalization of the mean value theorem
for a function with a convex derivative. We state the result in full
generality because it plays a role later.

\begin{lemma}[(Mean value trace inequality)] \label{lem:mvti}
Let $I$ be an interval of the real line. Suppose that $g \dvtx I \to\R
$ is
a weakly increasing function and that $h \dvtx I \to\R$ is a function
whose derivative $h'$ is convex. For all matrices $\mtx{A}, \mtx{B}
\in
\mathbb{H}^{d}(I)$, it holds that
\begin{eqnarray*}
&&\ntr \bigl[ \bigl(g(\mtx{A}) - g(\mtx{B})\bigr) \cdot\bigl(h(
\mtx{A}) - h(\mtx{B})\bigr) \bigr]
\\
&&\qquad \leq\tfrac{1}{2} \ntr \bigl[ \bigl(g(\mtx{A}) - g(\mtx{B})\bigr) \cdot(\mtx{A}
- \mtx{B}) \cdot\bigl(h'(\mtx{A}) + h'(\mtx{B})\bigr)
\bigr].
\end{eqnarray*}
When $h'$ is concave, the inequality is reversed. The same results hold
for the standard trace.
\end{lemma}

To prove Lemma~\ref{lem:mvti}, we require a trace inequality \cite{Pet94Survey-Certain},
Proposition~3, that follows from the definition
of a matrix function and the spectral theorem for Hermitian matrices.
%

\begin{prop}[(Generalized Klein inequality)] \label{prop:klein-inequality}
Let $u_1, \ldots,  u_n$ and $v_1,\ldots, v_n$ be real-valued functions on
an interval $I$ of the real line. Suppose %
%
\begin{equation}
\label{eqn:klein-hyp} \sum_{k} u_k(a)
v_k(b) \geq0 \qquad\mbox{for all $a, b \in I$.}
\end{equation}
Then
\[
\ntr \biggl[ \sum_{k} u_k(\mtx{A})
v_k(\mtx{B}) \biggr] \geq0 \qquad\mbox{for all $\mtx{A}, \mtx{B} \in
\mathbb{H}^{d}(I)$.}
\]
\end{prop}

With the generalized Klein inequality, %
we can establish Lemma~\ref{lem:mvti} by developing the appropriate
scalar inequality.\vadjust{\goodbreak}

\begin{pf*}{Proof of Lemma~\ref{lem:mvti}}
Fix $a, b \in I$. Since $g$ is weakly increasing, $(g(a) - g(b)) \cdot
(a - b)\geq0$. The fundamental theorem of calculus and the convexity
of $h'$ yield the estimate
%
\begin{eqnarray}
\label{eqn:scalar-mvi} %
&&\bigl(g(a) - g(b)\bigr) \cdot\bigl(h(a) - h(b)\bigr)
\nonumber
\\
&&\qquad= \bigl(g(a) - g(b)\bigr) \cdot(a - b) \int_0^1
h'\bigl(\tau a + (1-\tau) b\bigr) \,\mathrm{d} {\tau}
\nonumber
\\[-8pt]
\\[-8pt]
\nonumber
&&\qquad\leq\bigl(g(a) - g(b)\bigr) \cdot(a - b) \int_0^1
\bigl[ \tau\cdot h'(a) + (1-\tau) \cdot h'(b) \bigr]
\,\mathrm{d} {\tau}
\\
&&\qquad= \frac{1}{2} \bigl[ \bigl(g(a) - g(b)\bigr) \cdot(a-b) \cdot
\bigl(h'(a) + h'(b)\bigr) \bigr].\nonumber
\end{eqnarray}
The inequality is reversed when $h'$ is concave.

Bound \eqref{eqn:scalar-mvi} can be written in the form \eqref
{eqn:klein-hyp} by expanding the products and collecting terms
depending on $a$ into functions $u_k(a)$ and terms depending on $b$
into functions $v_k(b)$. Proposition~\ref{prop:klein-inequality} then
delivers a trace inequality, which can be massaged into the desired
form using the cyclicity of the trace and the fact that standard
functions of the same matrix commute.
We omit the algebraic details.
\end{pf*}

\begin{rem}
We must warn the reader that the proof of Lemma~\ref{lem:mvti}
succeeds because
the trace contains a product of \emph{three} terms involving \emph{two}
matrices. The obstacle to proving more general results is that we cannot
reorganize expressions like $\trace(\mtx{ABAB})$ and $\trace(\mtx
{ABC})$ at will.
\end{rem}

\subsection{\texorpdfstring{Bounding the derivative of the trace m.g.f.}
{Bounding the derivative of the trace m.g.f.}} \label{sec:d-trace-mgf}

The central result in this section applies the method of exchangeable pairs
and the mean value trace inequality
to bound the derivative of the trace m.g.f. in terms of the conditional variance.
This is the most important step in our theory on the exponential
concentration of random matrices.

%

%

\begin{lemma}[(The derivative of the trace m.g.f.)] \label{lem:mgf-derivative}
Suppose that $(\mtx{X}, \mtx{X}') \in\mathbb{H}^{d} \times\mathbb
{H}^{d}$ is a
matrix Stein pair,
and assume that $\mtx{X}$ is almost surely bounded in norm. %
Define the %
trace m.g.f. $m(\theta) \mathrell=\Expect\ntr\econst
^{\theta\mtx{X}}$.
Then
%
\begin{eqnarray}
m'(\theta) &\leq&\theta\cdot\Expect\ntr \bigl[ \mtx{
\Delta}_{\mtx{X}} \econst ^{\theta\mtx{X}} \bigr] \qquad\mbox{when $\theta\geq0$};
\label{eqn:m-prime-Delta+}
\\
m'(\theta) &\geq&\theta\cdot\Expect\ntr \bigl[ \mtx{
\Delta}_{\mtx{X}} \econst ^{\theta\mtx{X}} \bigr]\qquad \mbox{when $\theta\leq0$.}
\label{eqn:m-prime-Delta-}
\end{eqnarray}
The conditional variance $\mtxx{\Delta}_{\mtx{X}}$ is defined
in \eqref
{eqn:conditional-variance}.
\end{lemma}

\begin{pf}
We begin with the expression for the derivative of the trace m.g.f.,
%
\begin{equation}
\label{eqn:m-prime-v1} m'(\theta) = \Expect\ntr \biggl[ \frac{\mathrm{d}}{\mathrm{d}{\theta}}
\econst^{\theta\mtx{X}} \biggr] = \Expect\ntr \bigl[ \mtx{X} \econst^{\theta\mtx{X}}
\bigr].
\end{equation}
We can move the derivative inside the expectation because of
the dominated convergence theorem and the boundedness of $\mtx{X}$.

Apply the method of exchangeable pairs, Lemma~\ref{lem:exchange},
with the function $\mtx{F}(\mtx{X}) = \econst^{\theta\mtx{X}}$ to
reach an alternative representation of the derivative \eqref{eqn:m-prime-v1},
%
\begin{equation}
\label{eqn:m-prime-v2} m'(\theta) = \frac{1}{2\alpha} \Expect\ntr \bigl[
\bigl(\mtx{X} - \mtx{X}'\bigr) \bigl(\econst^{\theta\mtx{X}} -
\econst^{\theta\mtx{X}'} \bigr) \bigr].
\end{equation}
We have used the boundedness of $\mtx{X}$ to verify the regularity
condition \eqref{eqn:regularity-mep}.

Expression \eqref{eqn:m-prime-v2} %
is perfectly suited for an application of the mean value trace
inequality, Lemma~\ref{lem:mvti}.
First, assume that $\theta\geq0$, and consider the function
$h \dvtx s \mapsto\econst^{\theta s}$. The derivative $h' \dvtx s
\mapsto
\theta\econst^{\theta s}$ is convex, so Lemma~\ref{lem:mvti} implies that
\begin{eqnarray*}
m'(\theta) &\leq&\frac{\theta}{4\alpha} \Expect\ntr \bigl[ \bigl(\mtx{X}
- \mtx{X}' \bigr)^2 \cdot \bigl(\econst^{\theta\mtx{X}} +
\econst ^{\theta\mtx{X}'} \bigr) \bigr]
\\
&=& \frac{\theta}{2\alpha} \Expect\ntr \bigl[ \bigl(\mtx{X} - \mtx{X}'
\bigr)^2 \cdot\econst^{\theta\mtx{X}} \bigr]
\\
&=& \theta\cdot\Expect\ntr \biggl[ \frac{1}{2\alpha} \Expect \bigl[ \bigl(\mtx{X}
- \mtx{X}'\bigr)^2 \vertt Z \bigr] \cdot
\econst^{\theta\mtx{X}} \biggr]. %
\end{eqnarray*}
The second line follows from the fact that $(\mtx{X}, \mtx{X}')$ is an
exchangeable pair.
In the last line, we have used the boundedness of $\mtx{X}$ and $\mtx
{X}'$ to invoke the
pull-through property of conditional expectation.
Identify the conditional variance $\mtxx{\Delta}_{\mtx{X}}$, defined
in \eqref{eqn:conditional-variance},
to complete the argument.

The result for $\theta\leq0$ follows from an analogous argument. In
this case, we simply observe that the derivative of the function $h
\dvtx s
\mapsto\econst^{\theta s}$ is now concave, so the mean value trace
inequality, Lemma~\ref{lem:mvti}, produces a lower bound. The remaining
steps are identical.
\end{pf}

\begin{rem}[(Regularity conditions)]
To simplify the presentation, we have instated a boundedness assumption
in Lemma~\ref{lem:mgf-derivative}. All the examples we discuss satisfy
this requirement.
When $\mtx{X}$ is unbounded, Lemma~\ref{lem:mgf-derivative} still holds
provided that $\mtx{X}$ meets an integrability condition.
\end{rem}

\section{\texorpdfstring{Exponential concentration for bounded random matrices.}
{Exponential concentration for bounded random matrices}}\label
{sec:concentration-bdd}

We are now prepared to establish exponential concentration
inequalities. Our first major result demonstrates
that an almost-sure bound for the conditional variance yields
exponential tail bounds for the extreme eigenvalues of a random
Hermitian matrix.
We can also obtain estimates for the expectation of the extreme eigenvalues.

\begin{thmm}[(Concentration for bounded random matrices)] \label
{thmm:concentration-bdd}
Consider a matrix Stein pair $(\mtx{X}, \mtx{X}') \in\mathbb{H}^{d}
\times
\mathbb{H}^{d}$.
Suppose there exist nonnegative constants $c, v$ for which the
conditional variance \eqref{eqn:conditional-variance} of the pair satisfies
%
\begin{equation}
\label{eqn:comparison} \mtxx{\Delta}_{\mtx{X}} \psdle c \mtx{X} + v \Id\qquad\mbox {almost
surely}.
\end{equation}
Then, for all $t \geq0$,
\begin{eqnarray*}
\mathbb{P} \bigl\{ {\lambda_{\min}( \mtx{X} ) \leq-t } \bigr\} &\leq& d
\cdot\exp \biggl\{ \frac{-t^2}{2v} \biggr\},
\\
\mathbb{P} \bigl\{ {\lambda_{\max}( \mtx{X} ) \geq t } \bigr\} &\leq &d
\cdot\exp \biggl\{ -\frac{t}{c} + \frac{v}{c^2} \log \biggl(1 +
\frac{ct}{v} \biggr) \biggr\}
\\
&\leq& d \cdot\exp \biggl\{ \frac{-t^2}{2v + 2ct} \biggr\}.
\end{eqnarray*}
Furthermore,
\begin{eqnarray*}
\Expect\lambda_{\min}(\mtx{X}) &\geq&-\sqrt{2v\log d},
\\
\Expect\lambda_{\max}(\mtx{X}) &\leq& \sqrt{2v\log d} + c \log
d.
\end{eqnarray*}
\end{thmm}

This result may be viewed as a matrix analogue of Chatterjee's
concentration inequality for scalar random variables \cite{Cha07Steins-Method},
Theorem~1.5(ii).
The proof of Theorem~\ref{thmm:concentration-bdd} appears below in
Section~\ref{sec:proof-conc-bdd}.
Before we present the argument, let us explain how the result provides
a short proof of a Hoeffding-type
inequality for matrices.

%

%

%

%

%

%

%

\subsection{\texorpdfstring{Application: Matrix Hoeffding inequality.}
{Application: Matrix Hoeffding inequality}}

Theorem~\ref{thmm:concentration-bdd} yields an extension of Hoeffding's
inequality \cite{Hoeffding63}
that holds for an independent sum of bounded random matrices.

\begin{cor}[(Matrix Hoeffding)] \label{cor:hoeffding}
Consider a finite sequence $( \mtx{Y}_k )_{k \geq1}$ of independent
random matrices in $\mathbb{H}^{d}$
and a finite sequence $(\mtx{A}_k)_{k \geq1}$ of deterministic
matrices in $\mathbb{H}^{d}$. Assume that
\[
\Expect\mtx{Y}_k = \mtx{0}\quad \mbox{and}\quad \mtx{Y}_k^2
\psdle\mtx{A}_k^2 \qquad\mbox{almost surely for each index
$k$.}
\]
Then, for all $t\geq0$,
\[
\mathbb{P} \biggl\{ {\lambda_{\max} \biggl( \sum
_{k} \mtx{Y}_k \biggr) \geq t } \biggr\} \leq d
\cdot\econst^{-t^2/2\sigma^2} \qquad\mbox{for } \sigma^2 \mathrell=
\frac{1}{2} \biggl\llVert {\sum_k \bigl(\mtx
{A}_k^2 + \Expect\mtx{Y}_k^2
\bigr) } \biggr\rrVert.
\]
Furthermore,
\[
\Expect\lambda_{\max} \biggl( \sum_{k}
\mtx{Y}_k \biggr) \leq \sigma\sqrt{2 \log d}.
\]
\end{cor}

\begin{pf}
Let $\mtx{X} = \sum_k \mtx{Y}_k$. Since $\mtx{X}$ is a
sum of
centered, independent random matrices, we can use the matrix Stein pair
constructed in Section~\ref{sec:indep-sum}. According to \eqref
{eqn:indep-sum-DeltaX}, the conditional variance satisfies
\[
\mtxx{\Delta}_{\mtx{X}} = \frac{1}{2} \sum
_k \bigl( \mtx{Y}_k^2 + \Expect
\mtx{Y}_k^2 \bigr) \psdle\sigma^2 \Id
\]
because $\mtx{Y}_k^2 \psdle\mtx{A}_k^2$. Invoke Theorem~\ref
{thmm:concentration-bdd} with $c = 0$ and $v = \sigma^2$
to complete the bound.
\end{pf}

In the scalar setting $d = 1$, Corollary~\ref{cor:hoeffding} reproduces
an inequality of Chatterjee~\cite{Cha07Steins-Method}, Section~1.5,
which itself is an improvement over the classical scalar Hoeffding
bound. In turn, Corollary~\ref{cor:hoeffding}
improves upon the matrix Hoeffding inequality
of \cite{Tro11User-Friendly-FOCM},
Theorem~1.3, in two ways. First, we have
improved the constant in the exponent to its optimal value $1/2$.
Second, we have decreased the size of the variance measure because
$\sigma^2 \leq\llVert {\sum_k \mtx{A}_k^2 } \rrVert $.
Finally, let us
remark that a similar result holds under the weaker assumption that
$\sum_k \mtx{Y}_k^2 \psdle\mtx{A}^2$ almost surely.

Corollary~\ref{cor:hoeffding} admits a plethora %
of applications.
For example, in theoretical computer science, Widgerson and Xiao employ
a suboptimal matrix Hoeffding inequality \cite{WigdersonXi08},
Theorem~2.6, to derive efficient, derandomized
algorithms for homomorphism testing and semidefinite covering problems.
Under the improvements of Corollary~\ref{cor:hoeffding}, their results
improve accordingly.
%

%

%

%

%

%

%

\subsection{\texorpdfstring{Proof of Theorem \protect\ref{thmm:concentration-bdd}: Exponential concentration.}
{Proof of Theorem 4.1: Exponential concentration}} \label{sec:proof-conc-bdd}

Suppose that\break  $(\mtx{X}, \mtx{X}')$ is a matrix Stein pair %
constructed from an auxiliary exchangeable pair $(Z, Z')$.
Our aim is to bound the normalized trace m.g.f.
%
\begin{equation}
\label{eqn:tmgf-pf} m(\theta) \mathrell=
\Expect\ntr\econst^{\theta\mtx{X}}\qquad
\mbox{for $\theta\in\R$.}
\end{equation}
The basic strategy is to develop a differential inequality, which we
integrate to control $m(\theta)$ itself. Once these estimates
are in place, the matrix Laplace transform method,
Proposition~\ref{prop:matrix-laplace}, furnishes probability
inequalities for the extreme eigenvalues of $\mtx{X}$.

The following result summarizes our bounds for the trace m.g.f. $m(\theta)$.

\begin{lemma} [(Trace m.g.f. estimates for bounded random matrices)] \label
{lem:mgf-bounds}
Let $(\mtx{X}, \mtx{X}')$ be a matrix Stein pair, %
and suppose there exist nonnegative constants $c, v$ for which
%
\begin{equation}
\label{eqn:comparison2} \mtxx{\Delta}_{\mtx{X}} \psdle c \mtx{X} + v \Id \qquad\mbox{almost
surely}.
\end{equation}
Then the normalized trace m.g.f. $m(\theta) \mathrell=\Expect
\ntr\econst
^{\theta\mtx{X}}$ satisfies the bounds
%
\begin{eqnarray}
\log m(\theta) &\leq&\frac{v \theta^2}{2} \qquad\mbox{when $\theta\leq0$},
\label{eqn:mgf-negative}
\\
\log m(\theta) &\leq&\frac{v}{c^2} \biggl[ \log \biggl( \frac{1}{1-c\theta}
\biggr) - c\theta \biggr] \label{eqn:mgf-positive-1}
\\
&\leq&\frac{v \theta^2}{ 2(1 - c \theta)} \qquad\mbox{when $0 \leq\theta< 1/c$.} \label{eqn:mgf-positive-2}
\end{eqnarray}
\end{lemma}

We establish Lemma~\ref{lem:mgf-bounds} in Section~\ref{sec:L2-concentration} et seq. In Section~\ref{sec:stein-laplace-arg},
we finish the proof of Theorem~\ref{thmm:concentration-bdd} by
combining these bounds
with the matrix Laplace transform method.

\subsubsection{\texorpdfstring{Boundedness of the random matrix.}
{Boundedness of the random matrix}} \label{sec:L2-concentration}

First, we confirm that the random matrix $\mtx{X}$ is almost surely
bounded under hypothesis \eqref{eqn:comparison2} on the conditional variance
$\mtxx{\Delta}_{\mtx{X}}$. Recall definition \eqref
{eqn:conditional-variance} of
the conditional variance, and compute that
\[
\mtxx{\Delta}_{\mtx{X}} = \frac{1}{2\alpha} \Expect\bigl[ \bigl(\mtx{X} -
\mtx{X}'\bigr)^2 \vertt Z \bigr] \psdge\frac{1}{2\alpha}
\bigl(\Expect\bigl[ \mtx{X} - \mtx{X}' \vertt Z \bigr]
\bigr)^2 = \frac{1}{2\alpha} ( \alpha\mtx{X} )^2 =
\frac{\alpha}{2} \mtx{X}^2.
\]
The semidefinite bound is the operator Jensen inequality \eqref
{eqn:kadison}, applied conditionally.
The third relation follows from definition \eqref{eqn:F-and-G}
of a matrix Stein pair. Owing to assumption \eqref{eqn:comparison2},
we reach
the quadratic inequality
$\frac{1}{2} \alpha\mtx{X}^2 \psdle c \mtx{X} + v \Id$.
The scale factor $\alpha$ is positive, so we may conclude that the
eigenvalues of $\mtx{X}$
are almost surely restricted to a bounded interval.

%

\subsubsection{\texorpdfstring{Differential inequalities for the trace m.g.f.}
{Differential inequalities for the trace m.g.f.}}

Since the matrix $\mtx{X}$ is almost surely bounded, the derivative of
the trace m.g.f. has the form
%
\begin{equation}
\label{eqn:m-prime-pf} m'(\theta) %
= \Expect\ntr \bigl[ \mtx{X}
\econst^{\theta\mtx{X}} \bigr]\qquad \mbox{for $\theta\in\R$.}
\end{equation}
To control the derivative, we combine Lemma~\ref{lem:mgf-derivative}
with the assumed inequality~\eqref{eqn:comparison2} for the conditional
variance.
For $\theta\geq0$, we obtain
\begin{eqnarray*}
m'(\theta) &\leq&\theta\cdot\Expect\ntr \bigl[ \mtxx{
\Delta}_{\mtx{X}} \econst ^{\theta\mtx{X}} \bigr]
\\
&\leq&\theta\cdot\Expect\ntr \bigl[ (c \mtx{X} + v \Id) \econst ^{\theta\mtx{X}}
\bigr]
\\
&=& c\theta\cdot\Expect\ntr \bigl[ \mtx{X} \econst^{\theta\mtx{X}} \bigr] + v
\theta\cdot\Expect\ntr\econst^{\theta\mtx{X}}
\\
&=& c\theta\cdot m'(\theta) + v\theta\cdot m(\theta).
\end{eqnarray*}
In the last line, we have identified the trace m.g.f. \eqref{eqn:tmgf-pf}
and its derivative \eqref{eqn:m-prime-pf}.
The second relation holds because the matrix $\econst^{\theta\mtx{X}}$
is positive definite. Indeed, when $\mtx{P}$ is psd, $\mtx{A} \psdle
\mtx{B}$ implies that $\trace( \mtx{AP} ) \leq\trace( \mtx{BP} )$.

For $\theta\leq0$, the same argument yields a lower bound
\[
m'(\theta) \geq c\theta\cdot m'(\theta) + v\theta\cdot
m(\theta).
\]
Rearrange these inequalities to isolate the log-derivative $m'(\theta
)/m(\theta)$ of the trace m.g.f. We reach
%
\begin{eqnarray}
\frac{\mathrm{d}}{\mathrm{d}{\theta}} \log m(\theta) & \leq&\frac
{v \theta}{1 - c \theta}\qquad \mbox{for $0 \leq
\theta< 1/c$\quad and} \label{eqn:d-log-m-pos}
\\
\frac{\mathrm{d}}{\mathrm{d}{\theta}} \log m(\theta) & \geq&\frac
{v \theta}{1 - c \theta} \qquad\mbox{for $\theta
\leq0$} \label{eqn:d-log-m-neg}.
\end{eqnarray}

\subsubsection{\texorpdfstring{Solving the differential inequalities.}
{Solving the differential inequalities}}

Observe that
%
\begin{equation}
\label{eqn:m0} \log m(0) = \log\ntr\econst^{\mtx{0}} = \log\ntr\Id= \log1 =
0.
\end{equation}
Therefore, we may integrate the differential inequalities \eqref
{eqn:d-log-m-pos} and \eqref{eqn:d-log-m-neg}, starting at zero, to
obtain bounds on $\log m(\theta)$ elsewhere.

First, assume that $0 \leq\theta< 1/c$. In view of \eqref{eqn:m0},
the fundamental theorem of calculus and the differential
inequality \eqref{eqn:d-log-m-pos} imply that
\[
\log m(\theta) %
= \int_0^\theta
\frac{\mathrm{d}}{\mathrm{d}{s}} \log m(s) \,\mathrm{d} {s} \leq\int_0^\theta
\frac{v s}{1-cs}\,\mathrm{d} {s} = -\frac{v}{c^2} \bigl(c\theta+ \log(1-c
\theta) \bigr).
\]
We can develop a weaker inequality by making a further approximation
within the integral,
\[
\log m(\theta) %
\leq\int_0^\theta
\frac{vs}{1-cs} \,\mathrm{d} {s} \leq\int_0^\theta
\frac{vs}{1-c\theta} \,\mathrm{d} {s} = \frac{v\theta^2}{2(1-c\theta)}.
\]
These inequalities are the trace m.g.f. estimates \eqref{eqn:mgf-positive-1}
and \eqref{eqn:mgf-positive-2} appearing in Lemma~\ref{lem:mgf-bounds}.

Next, assume that $\theta\leq0$. In this case, the differential
inequality \eqref{eqn:d-log-m-neg} yields
\[
- \log m(\theta) %
= \int_\theta^0
\frac{\mathrm{d}}{\mathrm{d}{s}} \log m(s) \,\mathrm{d} {s} \geq\int_\theta^0
\frac{vs}{1-cs} \,\mathrm{d} {s} \geq\int_\theta^0
vs \,\mathrm{d} {s} = -\frac{v\theta^2}{2}.
\]
This calculation delivers the trace m.g.f. bound \eqref{eqn:mgf-negative}.
The proof of Lemma~\ref{lem:mgf-bounds} is complete.

%

\subsubsection{\texorpdfstring{The matrix Laplace transform argument.}
{The matrix Laplace transform argument}}
\label{sec:stein-laplace-arg}

With Lemma~\ref{lem:mgf-bounds} at hand, we quickly finish the proof of
Theorem~\ref{thmm:concentration-bdd}.
First, let us establish probability inequalities for the maximum
eigenvalue. The Laplace transform bound \eqref{eqn:laplace-upper-tail}
and the trace m.g.f. estimate \eqref{eqn:mgf-positive-1} together yield
\begin{eqnarray*}
\mathbb{P} \bigl\{ {\lambda_{\max}(\mtx{X}) \geq t } \bigr\}
&\leq&\inf_{0 < \theta< 1/c} d \cdot\exp \biggl\{ - \theta t -
\frac
{v}{c^2} \bigl( c\theta+ \log(1-c\theta) \bigr) \biggr\}
\\
&\leq& d \cdot\exp \biggl\{ - \frac{t}{c} + \frac{v}{c^2} \log
\biggl( 1 + \frac{ct}{v} \biggr) \biggr\}.
\end{eqnarray*}
The second relation follows when we choose $\theta= t/(v + ct)$.
Similarly, the trace m.g.f. bound \eqref{eqn:mgf-positive-2} delivers
\begin{eqnarray*}
\mathbb{P} \bigl\{ {\lambda_{\max}(\mtx{X}) \geq t } \bigr\} &\leq&\inf
_{0 < \theta< 1/c} d \cdot\exp \biggl\{ - \theta t + \frac{v
\theta^2}{2(1-c\theta)}
\biggr\}
\\
&=& d \cdot\exp \biggl\{ - \frac{v}{2c^2}(1-\sqrt{1+2ct/v})^2
\biggr\}
\\
&\leq& d \cdot\exp \biggl\{ - \frac{t^2}{2v + 2ct} \biggr\},
\end{eqnarray*}
because the infimum occurs at $\theta= (1-1/\sqrt{1+2ct/v})/c$.
The final inequality depends on the numerical fact
\[
(1-\sqrt{1+2x})^2 \geq\frac{x^2}{1+x} \qquad\mbox{for all $x \geq0$}.
\]
To control the expectation of the maximum eigenvalue, we combine the
Laplace transform bound \eqref{eqn:laplace-upper-mean} and the trace
m.g.f. bound \eqref{eqn:mgf-positive-2} to see that
\[
\Expect\lambda_{\max}(\mtx{X}) \leq\inf_{0 < \theta< 1/c}  \frac{1}{\theta} \biggl[ \log d + \frac
{v \theta^2}{2(1-c\theta)} \biggr] = \sqrt{2 v \log d}
+ c \log d.
\]
The second relation can be verified using a computer algebra system.

Next, we turn to results for the minimum eigenvalue. Combine the matrix
Laplace transform bound \eqref{eqn:laplace-lower-tail} with the trace
m.g.f. bound \eqref{eqn:mgf-negative} to reach
\[
\mathbb{P} \bigl\{ {\lambda_{\min}(\mtx{X}) \leq- t} \bigr\} \leq d
\cdot\inf_{\theta< 0}  \exp \biggl\{ \theta t + \frac{v
\theta
^2}{2}
\biggr\} = d \cdot\econst^{-t^2/2v}.
\]
The infimum is attained at $\theta= - t/v$. To compute the expectation
of the minimum eigenvalue, we apply the Laplace transform bound \eqref
{eqn:laplace-lower-mean} and the trace m.g.f. bound~\eqref
{eqn:mgf-negative}, whence
\[
\Expect\lambda_{\min}(\mtx{X}) \geq\sup_{\theta< 0}  \frac{1}{\theta} \biggl[ \log d + \frac{v
\theta^2}{2} \biggr] = - \sqrt{2 v\log d
}.
\]
The supremum is attained at $\theta= - \sqrt{2v^{-1}\log d}$.

%
\section{\texorpdfstring{Refined exponential concentration for random matrices.}
{Refined exponential concentration for random matrices}} \label
{sec:concentration-subgauss}

Although Theorem~\ref{thmm:concentration-bdd} is a strong result, the
hypothesis
$\mtxx{\Delta}_{\mtx{X}} \psdle c \mtx{X} + v \Id$ on the
conditional variance
is too stringent for many situations of interest.
Our second major result shows that we can use the typical
behavior of the conditional variance to obtain tail bounds for the
maximum eigenvalue of
a random Hermitian matrix.

%

%

%

%

%

%
\begin{thmm}[(Refined concentration for random matrices)] \label
{thmm:concentration-subgauss}
Suppose that $(\mtx{X}, \mtx{X}') \in\mathbb{H}^{d} \times\mathbb
{H}^{d}$ is a
matrix Stein pair,
and assume that $\mtx{X}$ is almost surely bounded in norm. Define the function
%
\begin{equation}
\label{eqn:r-psi} r(\psi): = \frac{1}{\psi} \log\Expect\ntr
\econst^{\psi\mtxx
{\Delta
}_{\mtx{X}}}\qquad \mbox{for each $\psi> 0$},
\end{equation}
where
$\mtxx{\Delta}_{\mtx{X}}$ is the conditional variance \eqref
{eqn:conditional-variance}.
Then, for all $t \geq0$ and all $\psi> 0$,
%
\begin{equation}
\label{eqn:refined-tail} \mathbb{P} \bigl\{ {\lambda_{\max}(\mtx{X}) \geq t} \bigr
\} \leq d \cdot\exp \biggl\{ \frac{-t^2}{2r(\psi) + 2 t /\sqrt{\psi}} \biggr\}.
\end{equation}
Furthermore, for all $\psi> 0$,
%
\begin{equation}
\label{eqn:refined-mean} \Expect\lambda_{\max}(\mtx{X}) \leq\sqrt{2 r(\psi) \log
d} + \frac{\log d}{\sqrt{\psi}}.
\end{equation}
\end{thmm}

This theorem is essentially a matrix version of a result from
Chatterjee's thesis~\cite{Cha08Concentration-Inequalities}, Theorem~3.13.
The proof of Theorem~\ref{thmm:concentration-subgauss} is similar in
spirit to
the proof of Theorem~\ref{thmm:concentration-bdd}, so we postpone
the demonstration until Appendix \ref{sec:subgauss-pf}.

Let us offer some remarks to clarify the meaning of this result.
Recall that $\mtxx{\Delta}_{\mtx{X}}$ is a stochastic approximation
for the variance of the random matrix $\mtx{X}$.
We can interpret the function $r(\psi)$ as a measure of the
typical magnitude of the conditional variance. Indeed, the matrix Laplace
transform result, Proposition~\ref{prop:matrix-laplace},
ensures that
\[
\Expect\lambda_{\max}( \mtxx{\Delta}_{\mtx{X}} ) \leq\inf
_{\psi> 0} \biggl[ r(\psi) + \frac{\log d}{\psi} \biggr].
\]
The import of this inequality is that we can often identify a
value of $\psi$ to make
$r(\psi) \approx\Expect\lambda_{\max}(\mtxx{\Delta}_{\mtx{X}})$.
Ideally, we also want to choose $r(\psi) \gg\psi^{-1/2}$ so that the term
$r(\psi)$ drives the tail bound \eqref{eqn:refined-tail}
when the parameter $t$ is small. In the next subsection, we show
that these heuristics yield a matrix Bernstein inequality.

%

%

%

%

\subsection{\texorpdfstring{Application: The matrix Bernstein inequality.}
{Application: The matrix Bernstein inequality}}

As an illustration of Theorem~\ref{thmm:concentration-subgauss}, we
establish a tail bound for a
sum of centered, independent random matrices that are subject to a
uniform norm bound.

\begin{cor}[(Matrix Bernstein)] \label{cor:bernstein}
Consider an independent sequence $(\mtx{Y}_k)_{k \geq1}$ of random
matrices in $\mathbb{H}^{d}$ that satisfy
\[
\Expect\mtx{Y}_k = \mtx{0}\quad \mbox{and}\quad \llVert {
\mtx{Y}_k } \rrVert \leq R\qquad \mbox{for each index $k$.}
\]
Then, for all $t \geq0$,
\[
\mathbb{P} \biggl\{ {\lambda_{\max} \biggl(\sum
_k \mtx{Y}_k \biggr) \geq t } \biggr\} \leq d
\cdot\exp \biggl\{ \frac{-t^2}{3 \sigma^2 + 2 R t} \biggr\} \qquad\mbox{for } \sigma^2
\mathrell=%
\biggl\llVert {\sum_k
\Expect\mtx{Y}_k^2 } \biggr\rrVert.
\]
Furthermore,
\[
\Expect\lambda_{\max} \biggl(\sum_k
\mtx{Y}_k \biggr) \leq\sigma\sqrt{3 \log d} + R \log d.
\]
\end{cor}

Corollary~\ref{cor:bernstein} is directly
comparable with other matrix Bernstein inequalities in the literature.
The constants are slightly
worse than \cite{Tro11User-Friendly-FOCM}, Theorem~1.4 and slightly
better than \cite{Oli10Concentration-Adjacency}, Theorem~1.2. The
hypotheses in the current result are somewhat stricter than those in
the prior works.
Nevertheless, the proof provides a template for studying more complicated
random matrices that involve dependent random variables.

\begin{pf*}{Proof of Corollary \ref{cor:bernstein}}
Consider the matrix Stein pair $(\mtx{X}, \mtx{X}')$ described in
Section~\ref{sec:indep-sum}.
Calculation \eqref{eqn:indep-sum-DeltaX} shows that the
conditional\vadjust{\goodbreak}
variance of $\mtx{X}$ satisfies
\[
\mtxx{\Delta}_{\mtx{X}} = \frac{1}{2} \sum
_k \bigl(\mtx {Y}_k^2 + \Expect
\mtx{Y}_k^2 \bigr).
\]
The function $r(\psi)$ measures the typical size of $\mtxx{\Delta
}_{\mtx
{X}}$. To control $r(\psi)$,
we center the conditional variance and reduce the expression as follows:
%
\begin{eqnarray}
\label{eqn:L-bd} r(\psi) \mathrell\hspace*{-1pt}&=&\frac{1}{\psi} \log\Expect\ntr
\econst^{\psi\mtxx
{\Delta}_{\mtx{X}}} \leq\frac{1}{\psi} \log\Expect\ntr\exp \bigl\{ \psi(
\mtx {\Delta }_{\mtx{X}} - \Expect\mtxx{\Delta}_{\mtx{X}}) + \psi\llVert
{\Expect\mtxx{\Delta}_{\mtx{X}} } \rrVert \cdot\Id \bigr\}
\nonumber
\\
&= &\frac{1}{\psi} \log\Expect\ntr \bigl[ \econst^{\psi\sigma^2} \cdot \exp
\bigl\{ \psi(\mtxx{\Delta}_{\mtx{X}} - \Expect\mtxx{\Delta }_{\mtx
{X}})
\bigr\} \bigr]
\\
&=& \sigma^2 + \frac{1}{\psi} \log\Expect\ntr
\econst^{ \psi(\mtxx{\Delta
}_{\mtx
{X}} - \Expect\mtxx{\Delta}_{\mtx{X}}) }.\nonumber
\end{eqnarray}
The inequality depends on the monotonicity of the trace
exponential \cite{Pet94Survey-Certain}, Section~2.
Afterward, we have applied the identity
$\llVert {\Expect\mtxx{\Delta}_{\mtx{X}} } \rrVert  = \llVert {\Expect\mtx {X}^2 } \rrVert  =
\sigma^2$,
which follows from \eqref{eqn:mean-delta} and the independence of the
sequence $(\mtx{Y}_k)_{k\geq1}$.

Introduce the centered random matrix
%
\begin{equation}
\label{eqn:W-def-bernstein} \mtx{W} \mathrell=\mtxx{\Delta}_{\mtx{X}} - \Expect
\mtx {\Delta }_{\mtx{X}} = \frac{1}{2} \sum
_k \bigl(\mtx{Y}_k^2 - \Expect
\mtx{Y}_k^2 \bigr).
\end{equation}
Observe that $\mtx{W}$ consists of a sum of centered, independent
random matrices,
so we can study it using the matrix Stein pair discussed in
Section~\ref{sec:indep-sum}.
Adapt the conditional variance calculation \eqref{eqn:indep-sum-DeltaX}
to obtain
\begin{eqnarray*}
\mtxx{\Delta}_{\mtx{W}} &=& \frac{1}{2} \cdot\frac{1}{4} \sum
_k \bigl[ \bigl(\mtx{Y}_k^2
- \Expect\mtx{Y}_k^2 \bigr)^2 + \Expect
\bigl(\mtx{Y}_k^2 - \Expect\mtx{Y}_k^2
\bigr)^2 \bigr]
\\
&\psdle&\frac{1}{8} \sum_k \bigl[ 2
\mtx{Y}_k^4 + 2 \bigl(\Expect\mtx{Y}_k^2
\bigr)^2 + \Expect \mtx{Y}_k^4 - \bigl(
\Expect\mtx{Y}_k^2 \bigr)^2 \bigr]
\\
&\psdle&\frac{1}{4} \sum_k \bigl(
\mtx{Y}_k^4 + \Expect \mtx {Y}_k^4
\bigr).
\end{eqnarray*}
To reach the second line, we apply the operator convexity \eqref
{eqn:square-convex} of the matrix square to the first parenthesis, and
we compute the second expectation explicitly. The third line follows
from the operator Jensen inequality \eqref{eqn:kadison}. To continue,
make the estimate $\mtx{Y}_k^4 \psdle R^2 \mtx{Y}_k^2$ in both
terms. Thus,
\[
\label{eqn:DeltaW-bd} \mtxx{\Delta}_{\mtx{W}} \psdle\frac{R^2}{4} \sum
_{k=1}^n \bigl(\mtx{Y}_k^2
+ \Expect \mtx{Y}_k^2 \bigr) \psdle\frac{R^2}{2}
\cdot\mtx{W} + \frac{R^2 \sigma^2}{2} \cdot \Id.
\]
The trace m.g.f. bound, Lemma~\ref{lem:mgf-bounds}, delivers
%
\begin{equation}
\label{eqn:center-DeltaX-bd} \log m_{\mtx{W}}(\psi) = \log\Expect\ntr
\econst^{\psi\mtx{W}} \leq\frac{R^2 \sigma^2 \psi^2}{4 - 2 R^2 \psi}.
\end{equation}

To complete the proof, combine the bounds \eqref{eqn:L-bd} and \eqref
{eqn:center-DeltaX-bd} to reach
\[
r(\psi) %
\leq\sigma^2 + \frac{R^2 \sigma^2 \psi}{4 - 2 R^2 \psi}.
\]
In particular, it holds that $r(R^{-2}) \leq1.5 \sigma^2$. The
result now
follows from Theorem~\ref{thmm:concentration-subgauss}.
\end{pf*}

%

\section{\texorpdfstring{Polynomial moments and the spectral
norm of a random matrix.}
{Polynomial moments and the spectral norm of a random matrix}}
\label{sec:polynomial}

We can also study the spectral norm of a random matrix by bounding its
polynomial moments.
To present
these results, we must introduce
the family of Schatten norms.

\begin{defn}[(Schatten norm)]
For each $p \geq1$, the Schatten $p$-norm is defined as
\[
\llVert {\mtx{B}} \rrVert _{p} \mathrell= \bigl( \trace
\llvert {\mtx{B}} \rrvert ^p \bigr)^{1/p} \qquad\mbox{for $\mtx{B}
\in\M^d$.}
\]
In this setting, $\llvert {\mtx{B}} \rrvert  \mathrell=(\mtx{B}^*\mtx{B})^{1/2}$.
Bhatia's book \cite{Bha97Matrix-Analysis}, Chapter IV, contains a
detailed discussion
of these norms and their properties.
\end{defn}

The following proposition is a matrix analog of the Chebyshev bound
from classical probability.
As in the scalar case \cite{Lug09Concentration-Measure}, Exercise 1,
this bound is at least as tight as the analogous matrix Laplace
transform bound \eqref{eqn:laplace-upper-tail}.

%

%

\begin{prop}[(Matrix Chebyshev method)] \label{prop:chebyshev}
Let $\mtx{X}$ be a random matrix.
For all $t > 0$,
%
\begin{equation}
\label{eqn:chebyshev-tail} \mathbb{P} \bigl\{ {\llVert {\mtx{X}} \rrVert \geq t } \bigr\}
\leq\inf_{p \geq1}  t^{-p} \cdot\Expect\llVert {\mtx{X} }
\rrVert _{p}^p. %
\end{equation}
Furthermore,
%
\begin{equation}
\label{eqn:chebyshev-mean} \Expect\llVert {\mtx{X}} \rrVert \leq\inf_{p \geq1}
 \bigl( \Expect\llVert {\mtx{X} } \rrVert _{p}^p
\bigr)^{1/p}. %
\end{equation}
\end{prop}

\begin{pf}
To prove \eqref{eqn:chebyshev-tail}, we use Markov's inequality. For $p
\geq1$,
\[
\mathbb{P} \bigl\{ {\llVert {\mtx{X}} \rrVert \geq t } \bigr\} %
\leq
t^{-p} \cdot\Expect\llVert {\mtx{X} } \rrVert ^p =
t^{-p} \cdot\Expect\bigl\llVert {\llvert {\mtx{X}} \rrvert
^p } \bigr\rrVert \leq t^{-p} \cdot\Expect\trace\llvert {
\mtx{X}} \rrvert ^p,
\]
since the trace of a positive matrix dominates the maximum eigenvalue.
To verify~\eqref{eqn:chebyshev-mean}, select $p \geq1$. Jensen's
inequality implies that
\[
\Expect\llVert {\mtx{X} } \rrVert \leq \bigl( \Expect\llVert {\mtx{X} } \rrVert
^p \bigr)^{1/p} = \bigl( \Expect\bigl\llVert {\llvert {
\mtx{X}} \rrvert ^p } \bigr\rrVert \bigr)^{1/p} \leq \bigl(
\Expect\trace\llvert {\mtx{X}} \rrvert ^p \bigr)^{1/p}.
\]
Identify the Schatten $p$-norm and take infima to complete the bounds.
\end{pf}

%

%

\section{\texorpdfstring{Polynomial moment inequalities for random matrices.}
{Polynomial moment inequalities for random matrices}} \label
{sec:burkholder}

Our last major result demonstrates that the polynomial moments of a
random Hermitian matrix are controlled by the moments of the\vadjust{\goodbreak}
conditional variance. By combining this result with the matrix
Chebyshev method, Proposition~\ref{prop:chebyshev}, we can obtain
probability inequalities for the spectral norm of a random Hermitian matrix.

\begin{thmm} [(Matrix BDG inequality)] \label{thmm:BDG-inequality}
Let $p = 1$ or $p \geq1.5$.
Suppose that $(\mtx{X}, \mtx{X}')$ is a matrix Stein pair where
$\Expect\llVert {\mtx{X}} \rrVert _{2p}^{2p} < \infty$.
Then
\[
\bigl( \Expect\llVert {\mtx{X}} \rrVert _{2p}^{2p}
\bigr)^{1/(2p)} \leq\sqrt{2p-1} \cdot \bigl( \Expect\llVert {\mtxx{\Delta
}_{\mtx{X}} } \rrVert _{p}^p \bigr)^{1/(2p)}.
\]
The conditional variance $\mtxx{\Delta}_{\mtx{X}}$ is defined
in \eqref
{eqn:conditional-variance}.
\end{thmm}
%
\begin{rem}[(Missing values)]
Theorem~\ref{thmm:BDG-inequality} also holds when $1 < p < 1.5$. In
this range, our bound for the constant is $\sqrt{4p-2}$. The proof
requires a variant of the mean value trace inequality for a convex
function $h$.
\end{rem}

Theorem~\ref{thmm:BDG-inequality} extends a scalar result of
Chatterjee \cite{Cha07Steins-Method}, Theorem~1.5(iii), to the matrix
setting. Chatterjee's bound can be viewed as an exchangeable pairs
version of the Burkholder--Davis--Gundy (BDG) inequality from classical
martingale theory \cite{Bur73Distribution-Function}. Other matrix
extensions of the BDG inequality appear in the work of Pisier--Xu \cite
{PX97Noncommutative-Martingale} and the work of Junge--Xu \cite
{JX03Noncommutative-Burkholder,JX08Noncommutative-Burkholder-II}.
The proof of Theorem~\ref{thmm:BDG-inequality}, which applies equally
to infinite dimensional operators $\mtx{X}$, appears below in
Section~\ref{sec:BDG-proof}.

%

%

\subsection{\texorpdfstring{Application: Matrix Khintchine inequality.}
{Application: Matrix Khintchine inequality}} \label{sec:khintchine}

First, we demonstrate that the matrix BDG inequality contains an
improvement of the noncommutative Khintchine inequality \cite
{L-P86Inegalites-Khintchine,LPP91Noncommutative-Khintchine} in the
matrix setting. This result has been a dominant tool in several
application areas over the last few years, largely because of the
articles~\cite{Rud99Random-Vectors,RV07Sampling-Large}.

\begin{cor}[(Matrix Khintchine)] \label{cor:khintchine}
Suppose that $p = 1$ or $p \geq1.5$.
Consider a finite sequence $( \mtx{Y}_k )_{k \geq1}$ of independent,
random, Hermitian matrices
and a deterministic sequence $( \mtx{A}_k )_{k \geq1}$ for which
%
\begin{equation}
\label{eqn:Zk-bound} \Expect\mtx{Y}_k = \mtx{0}\quad \mbox{and}\quad
\mtx{Y}_k^2 \psdle\mtx{A}_k^2\qquad
\mbox{almost surely for each index $k$.}
\end{equation}
Then
\[
\biggl( \Expect\biggl\llVert {\sum_k
\mtx{Y}_k } \biggr\rrVert _{2p}^{2p}
\biggr)^{1/(2p)} %
\leq\sqrt{p-0.5} \cdot\biggl\llVert { \biggl(
\sum_k \bigl( \mtx {A}_k^2
+ \Expect\mtx{Y}_k^2 \bigr) \biggr)^{1/2} }
\biggr\rrVert _{2p}.
\]
In particular, when $(\eps_k)_{k \geq1}$ is an independent sequence of
Rademacher random variables,
%
\begin{equation}
\label{eqn:nc-khintchine} \biggl( \Expect\biggl\llVert {\sum_k
\eps_k \mtx{A}_k } \biggr\rrVert _{2p}^{2p}
\biggr)^{1/(2p)} \leq\sqrt{2p-1} \cdot\biggl\llVert { \biggl( \sum
_k \mtx{A}_k^2
\biggr)^{1/2} } \biggr\rrVert _{2p}.
\end{equation}
\end{cor}

\begin{pf}
Consider the random matrix $\mtx{X} = \sum_k \mtx{Y}_k$. We
use the matrix Stein pair constructed in Section~\ref{sec:indep-sum}.
According to \eqref{eqn:indep-sum-DeltaX}, the conditional variance
$\mtxx{\Delta}_{\mtx{X}}$ satisfies
\[
\mtxx{\Delta}_{\mtx{X}} = \frac{1}{2} \sum
_k \bigl( \mtx{Y}_k^2 + \Expect
\mtx{Y}_k^2 \bigr) \psdle\frac{1}{2} \sum
_k \bigl( \mtx{A}_k^2 + \Expect
\mtx {Y}_k^2 \bigr). %
\]
An application of Theorem~\ref{thmm:BDG-inequality} completes the argument.
\end{pf}

For
each positive integer $p$, the optimal constant $\mathrm{C}_{2p}$
on the right-hand side of~\eqref{eqn:nc-khintchine} satisfies
%
\[
\cnst{C}_{2p}^{2p} = (2p-1)!! = {(2p)!/} {
\bigl(2^p p!\bigr)}
\]
as shown by Buchholz \cite{Buc01Operator-Khintchine}, Theorem~5.
Since
$
{(2p-1)^p/}{(2p-1)!!}
< \econst^{p-1/2}
$
for each positive integer $p$, the constant in \eqref
{eqn:nc-khintchine} lies within a factor $\sqrt{\econst}$ of optimal.
Previous methods for establishing the matrix Khintchine inequality are
rather involved, so it is
remarkable that the simple argument based on exchangeable pairs leads
to a result that is so accurate.
The same argument even yields a result under the weaker assumption
that $\sum_k \mtx{Y}_k^2 \psdle\mtx{A}^2$ almost surely.
%

\subsection{\texorpdfstring{Application: Matrix Rosenthal inequality.}
{Application: Matrix Rosenthal inequality}} \label{sec:rosenthal}

As a second example, we can develop a more sophisticated set of moment
inequalities that are roughly
the polynomial equivalent of the exponential moment bound underlying
the matrix Bernstein
inequality.

%

\begin{cor}[(Matrix Rosenthal inequality)] \label{cor:ros-pin}
Suppose that $p = 1$ or \mbox{$p \geq1.5$}. Consider a finite sequence $(\mtx
{P}_k)_{k \geq1}$ of independent, random
psd matrices that satisfy $\Expect\llVert {\mtx{P}_k} \rrVert _{2p}^{2p} <
\infty
$. Then
%
\begin{eqnarray}
\label{eqn:rosenthal-pos} %
&& \biggl( \Expect\biggl\llVert {\sum
_k \mtx{P}_k } \biggr\rrVert
_{2p}^{2p} \biggr)^{1/(2p)}
\nonumber
\\[-8pt]
\\[-8pt]
\nonumber
&&\qquad\leq \biggl[ \biggl\llVert {\sum_k
\Expect\mtx{P}_k } \biggr\rrVert _{2p}^{1/2} +
\sqrt{4p-2} \cdot \biggl( \sum_k \Expect\llVert {
\mtx {P}_k } \rrVert _{2p}^{2p}
\biggr)^{1/(4p)} \biggr]^2.
\end{eqnarray}
Now, consider a finite sequence $( \mtx{Y}_k )_{k \geq1}$ of centered,
independent, random Hermitian matrices,
and assume that $\Expect\llVert {\mtx{Y}_k} \rrVert
_{4p}^{4p} < \infty$. Then
%
\begin{eqnarray}\qquad
\label{eqn:rosenthal} %
&& \biggl( \Expect\biggl\llVert {\sum
_k \mtx{Y}_k } \biggr\rrVert
_{4p}^{4p} \biggr)^{1/(4p)}
\nonumber
\\[-8pt]
\\[-8pt]
\nonumber
&&\qquad\leq\sqrt{4p-1} \cdot\biggl\llVert { \biggl( \sum
_k \Expect\mtx {Y}_k^2
\biggr)^{1/2} } \biggr\rrVert _{4p} + (4p-1) \cdot \biggl( \sum
_k \Expect\llVert {\mtx{Y}_k }
\rrVert _{4p}^{4p} \biggr)^{1/(4p)}.
\end{eqnarray}
\end{cor}

Turn to Appendix \ref{sec:pf-rosenthal} for the proof of
Corollary~\ref
{cor:ros-pin}.
This result extends a moment inequality due to
Nagaev and Pinelis \cite{NP77Some-Inequalities},
which refines the constants in Rosenthal's inequality \cite{Ros70Subspaces-Lp},
Lemma~1.
See the historical discussion \cite{Pin94Optimum-Bounds}, Section~5,
for details.
An interesting application of Corollary~\ref{cor:ros-pin} is to
establish improved sample complexity bounds for masked sample covariance
estimation \cite{CGT11Masked-Sample} when the dimension of a
covariance matrix exceeds the number of samples.
As we were finishing this paper, we learned that Junge and Zheng have recently
established a noncommutative moment inequality \cite{JZ11Noncommutative-Bennett},
Theorem~0.4,
that is quite similar to Corollary~\ref{cor:ros-pin}.
%

%

\subsection{\texorpdfstring{Proof of the matrix BDG inequality.}
{Proof of the matrix BDG inequality}} \label{sec:BDG-proof}

In many respects, the proof of the matrix BDG inequality is similar to
the proof of the exponential concentration result, Theorem~\ref
{thmm:concentration-bdd}. Both are based on moment comparison arguments
that ultimately depend on the method of exchangeable pairs and the mean
value trace inequality.

Suppose that $(\mtx{X},\mtx{X}')$ is a matrix Stein pair with scale
factor $\alpha$.
First, observe that the result for $p = 1$ already follows from \eqref
{eqn:mean-delta}. Therefore, we may assume that $p \geq1.5$.
Introduce notation for the quantity of interest,
\[
E \mathrell=\Expect\llVert {\mtx{X}} \rrVert _{2p}^{2p}
= \Expect\trace\llvert {\mtx {X}} \rrvert ^{2p}.
\]
Rewrite the expression for $E$ by peeling off a copy of $\llvert
{\mtx {X}} \rrvert $. This move yields
\[
E = \Expect\trace \bigl[ \llvert {\mtx{X}} \rrvert \cdot \llvert {\mtx{X}}
\rrvert ^{2p-1} \bigr] = \Expect\trace \bigl[ \mtx{X} \cdot\sgn(\mtx{X})
\cdot\llvert {\mtx {X}} \rrvert ^{2p-1} \bigr].
\]
Apply the method of exchangeable pairs, Lemma~\ref{lem:exchange}, with
$\mtx{F}(\mtx{X}) = \sgn(\mtx{X}) \cdot\llvert {\mtx{X}}
\rrvert ^{2p - 1}$
to reach
\[
E = \frac{1}{2\alpha} \Expect\trace \bigl[ \bigl(\mtx{X} - \mtx{X}'
\bigr) \cdot \bigl( \sgn(\mtx{X}) \cdot\vert {\mtx{X}} \vert ^{2p-1} -
\operatorname{sgn} \bigl(\mtx{X}'\bigr) \cdot \bigl\vert {\mtx{X}'}
\bigr\vert ^{2p-1} \bigr) \bigr].
\]
To verify the regularity condition \eqref{eqn:regularity-mep} in
Lemma~\ref{lem:exchange}, compute that
\begin{eqnarray*}
&&\Expect{ \bigl\Vert{\bigl(\mtx{X} - \mtx{X}'\bigr) \cdot
\sgn(\mtx{X}) \cdot\llvert {\mtx{X}} \rrvert ^{2p-1} }
\bigr\Vert}_{}
\\
&&\qquad\leq\Expect \bigl( \llVert {\mtx{X}} \rrVert \llVert {\mtx{X} } \rrVert
^{2p-1} \bigr) + \Expect \bigl( \bigl\llVert {\mtx{X}'}
\bigr\rrVert \llVert {\mtx{X}} \rrVert ^{2p-1} \bigr)
\\
&&\qquad\leq2 \bigl(\Expect\llVert {\mtx{X}} \rrVert ^{2p}
\bigr)^{1/(2p)} \bigl(\Expect \llVert {\mtx{X} } \rrVert ^{2p}
\bigr)^{(2p-1)/2p}
\\
&&\qquad= 2 \Expect\llVert {\mtx{X} } \rrVert ^{2p} < \infty.
\end{eqnarray*}
We have used the fact that $\sgn(\mtx{X})$ is a unitary matrix, the
exchangeability of $(\mtx{X}, \mtx{X}')$, H\"{o}lder's inequality for
expectation and the fact that the Schatten $2p$-norm dominates the
spectral norm.

We intend to apply the mean value trace inequality to obtain an
estimate for the quantity $E$. Consider the function $h \dvtx s \mapsto
\sgn
(s) \cdot\llvert {s} \rrvert ^{2p-1}$. Its derivative $h'(s)
= (2p-1) \cdot\llvert {s} \rrvert ^{2p-2}$ is convex because
$p \geq1.5$. Lemma~\ref{lem:mvti}
delivers the bound
\begin{eqnarray*}
E & \leq&\frac{2p-1}{4\alpha} \Expect\trace \bigl[ \bigl( \mtx{X} -
\mtx{X}' \bigr)^2 \cdot \bigl( \llvert {\mtx{X}} \rrvert
^{2p-2} + \bigl\llvert {\mtx {X}'} \bigr\rrvert
^{2p-2} \bigr) \bigr]
\\
&=& \frac{2p-1}{2\alpha} \Expect\trace \bigl[ \bigl(\mtx{X} - \mtx{X}'
\bigr)^2 \cdot\llvert {\mtx{X}} \rrvert ^{2p-2} \bigr]
\\
&=& (2p-1) \cdot\Expect\trace \bigl[ \mtxx{\Delta}_{\mtx{X}} \cdot \llvert {
\mtx{X}} \rrvert ^{2p-2} \bigr]. 
\end{eqnarray*}
The second line follows from the exchangeability of $\mtx{X}$ and
$\mtx{X}'$.
In the last line, we identify the conditional variance $\mtxx{\Delta
}_{\mtx{X}}$, defined in \eqref{eqn:conditional-variance}. As before,
the moment bound $\Expect\llVert {\mtx{X}} \rrVert
_{2p}^{2p} < \infty$ is strong
enough to justify using the pull-through property in this step.

To continue, we must find a copy of $E$ within the latter expression.
We can accomplish this goal using one of the basic results from the
theory of Schatten norms \cite{Bha97Matrix-Analysis}, Corollary IV.2.6.

\begin{prop}[(H\"{o}lder inequality for trace)] \label{prop:holder-trace}
Let $p$ and $q$ be H\"{o}lder conjugate indices, that is, positive numbers
with the relationship \mbox{$q = p/(p-1)$}. Then
\[
\trace(\mtx{BC}) \leq\llVert {\mtx{B}} \rrVert _{p} \llVert {\mtx{C}}
\rrVert _{q} \qquad\mbox{for all $\mtx{B}, \mtx{C} \in\M^{d}$.}
\]
\end{prop}

To complete the argument, apply the H\"{o}lder inequality for the trace
followed by the H\"{o}lder inequality for the expectation. Thus
\begin{eqnarray*}
E &\leq&(2p-1) \cdot\Expect \bigl[ \llVert {\mtxx{\Delta}_{\mtx
{X}} } \rrVert
_{p} \cdot { \bigl\Vert{\llvert {\mtx{X}} \rrvert ^{2p-2} } \bigr\Vert
}_{p/(p-1)} \bigr]
\\
&=& (2p-1) \cdot\Expect \bigl[ \llVert {\mtxx{\Delta}_{\mtx{X}} } \rrVert
_{p} \cdot \llVert {\mtx{X} } \rrVert _{2p}^{2p-2}
\bigr]
\\
&\leq&(2p-1) \cdot \bigl( \Expect\llVert {\mtxx{\Delta}_{\mtx
{X}} } \rrVert
_{p}^p \bigr)^{1/p} \cdot \bigl( \Expect\llVert {
\mtx{X} } \rrVert _{2p}^{2p} \bigr)^{(p-1)/p}
\\
&= &(2p-1) \cdot \bigl( \Expect\llVert {\mtxx{\Delta}_{\mtx{X}} } \rrVert
_{p}^p \bigr)^{1/p} \cdot E^{(p-1)/p}.
\end{eqnarray*}
Solve this algebraic inequality for the positive number $E$ to conclude that
\[
E \leq(2p-1)^p \cdot\Expect\llVert {\mtxx{\Delta}_{\mtx{X}} }
\rrVert _{p}^p.
\]
Extract the $(2p)$th root to establish the matrix BDG inequality.

\section{\texorpdfstring{Extension to general complex matrices.}
{Extension to general complex matrices}} \label{sec:rectangular}

Although, at first sight, it may seem that our theory is limited to
random Hermitian matrices,
results for general random matrices follow as a formal corollary \cite
{Rec11Simpler-Approach,Tro11User-Friendly-FOCM}.
The approach is based on a device from operator theory \cite
{Pau02Completely-Bounded}.

\begin{defn}[(Hermitian dilation)]
Let $\mtx{B}$ be a matrix in $\C^{d_1 \times d_2}$, and set $d = d_1 +
d_2$. The \emph{Hermitian dilation} of $\mtx{B}$ is the matrix
\[
\mathscr{D} (\mtx{B}): = %
\lleft[\matrix{ \zeromtx&
\mtx{B} \vspace*{2pt}
\cr
\mtx{B}^* & \zeromtx} \rright] %
\in
\mathbb{H}^{d}.
\]
\end{defn}

The dilation has two valuable properties. First, it preserves spectral
information,
%
\begin{equation}
\label{eqn:dilation-spectrum} \lambda_{\max} \bigl(\mathscr{D}(\mtx{B})\bigr) = \bigl
\llVert {\mathscr{D} (\mtx {B})} \bigr\rrVert = \llVert {\mtx{B}} \rrVert.
\end{equation}
Second, the square of the dilation satisfies
%
\begin{equation}
\label{eqn:dilation-square} \mathscr{D}(\mtx{B})^2 = %
\lleft[
\matrix{ \mtx{BB}^*& \mtx{0} \vspace*{2pt}
\cr
\mtx{0} & \mtx{B}^* \mtx {B} }
\rright].
\end{equation}

We can study a random matrix---not necessarily Hermitian---%
by applying our matrix concentration inequalities to the Hermitian dilation
of the random matrix. As an illustration, let us prove a Bernstein inequality
for general random matrices.

%

\begin{cor}[(Bernstein inequality for general matrices)] \label{cor:bern-gen}
Consider a finite sequence $(\mtx{Z}_k)_{k \geq1}$ of independent
random matrices in $\C^{d_1 \times d_2}$ that satisfy
\[
\Expect\mtx{Z}_k = \mtx{0} \quad\mbox{and}\quad \llVert {
\mtx{Z}_k } \rrVert \leq R \qquad\mbox{almost surely for each index $k$.}
\]
Define $d \mathrell=d_1 + d_2$, and introduce the variance measure
\[
\sigma^2 \mathrell=\max \biggl\{ \biggl\llVert
{\sum_k \Expect\bigl( \mtx{Z}_k
\mtx{Z}_k^*\bigr) } \biggr\rrVert,  \biggl\llVert {\sum
_k \Expect\bigl( \mtx{Z}_k^*
\mtx{Z}_k \bigr) } \biggr\rrVert \biggr\}.
\]
Then, for all $t \geq0$,
%
\begin{equation}
\label{eqn:bern-gen-tail} \mathbb{P} \biggl\{ {\biggl\llVert {\sum
_k \mtx{Z}_k } \biggr\rrVert \geq t } \biggr\}
\leq d \cdot\exp \biggl\{ \frac{-t^2}{3 \sigma^2 + 2 R t} \biggr\}.
\end{equation}
Furthermore,
%
\begin{equation}
\label{eqn:bern-gen-mean} \Expect\biggl\llVert {\sum_k
\mtx{Z}_k } \biggr\rrVert \leq\sigma\sqrt{3 \log d} + R \log d.
\end{equation}
\end{cor}
\begin{pf}
Consider the random series
$\sum_k \mathscr{D}(\mtx{Z}_k)$.
The summands are independent, random Hermitian matrices that satisfy
\[
\Expect\mathscr{D}(\mtx{Z}_k) = \mtx{0} \quad\mbox{and}\quad \bigl\llVert {
\mathscr{D}(\mtx{Z}_k) } \bigr\rrVert \leq R.
\]
The second identity depends on the spectral property \eqref
{eqn:dilation-spectrum}.
Therefore, the matrix Bernstein inequality, Corollary~\ref
{cor:bernstein}, applies.
To state the outcome, we first note that
$\lambda_{\max}(\sum_k \mathscr{D}(\mtx{Z}_k)) = \llVert
{\sum_k \mtx{Z}_k } \rrVert $,
again because of the spectral property \eqref{eqn:dilation-spectrum}.
Next, use the formula \eqref{eqn:dilation-square} to compute that
\[
\biggl\llVert {\sum_k \Expect \bigl[
\mathscr{D}(\mtx{Z}_k)^2 \bigr] } \biggr\rrVert =
\left\Vert {\lleft[\matrix{ \displaystyle\sum_k \Expect
\bigl( \mtx{Z}_k \mtx {Z}_k^* \bigr) & \mtx{0}
\vspace*{2pt}
\cr
\mtx{0} & \displaystyle\sum_k \Expect\bigl(
\mtx{Z}_k^*\mtx{Z}_k \bigr) } \rright] } \right
\Vert = \sigma^2.
\]
This observation completes the proof.\vadjust{\goodbreak}
\end{pf}

Corollary~\ref{cor:bern-gen} has important implications for the problem
of estimating a matrix from noisy measurements.
Indeed, bound %
\eqref{eqn:bern-gen-mean} leads to a sample complexity analysis for
matrix completion \cite{FoygelSr11}.
Moreover, a variety of authors have used tail bounds of the form \eqref
{eqn:bern-gen-tail} to control the error of convex optimization methods
for matrix estimation %
\cite
{Gro11Recovering-Low-Rank,Rec11Simpler-Approach,NegahbanWa12,MackeyTaJo11}.

%

\section{\texorpdfstring{A sum of conditionally independent, zero-mean matrices.}
{A sum of conditionally independent, zero-mean matrices}}
\label
{sec:cond-zero-mean}

A chief advantage of the method of exchangeable pairs is its ability to
handle random matrices constructed from \emph{dependent} random
variables. In this section, we briefly describe a way to relax the
independence requirement
when studying a sum of random matrices. In Sections~\ref{sec:combinatorial-sum} and \ref{sec:self-repro}, we develop more
elaborate examples.

\subsection{\texorpdfstring{Formulation.}{Formulation}}
Let us consider a finite sequence $( \mtx{Y}_1, \ldots, \mtx{Y}_n )$ of
random Hermitian matrices that are conditionally independent given an
auxiliary random element $Z$.
Suppose moreover that
%
\begin{equation}
\label{eqn:cond-zero-mean} \Expect[\mtx{Y}_k \vertt Z ] = \mtx{0} \qquad\mbox{almost
surely for each index $k$.}
\end{equation}
We are interested in the sum of these conditionally independent,
zero-mean random matrices
%
\begin{equation}
\label{eqn:cond-zero-mean-sum} \mtx{X} \mathrell=\mtx{Y}_1 + \cdots+
\mtx{Y}_n.
\end{equation}
This type of series includes many examples that arise in practice.

\begin{example}[(Rademacher series with random matrix coefficients)]
Consider a finite sequence $(\mtx{W}_k)_{k \geq1}$ of random Hermitian
matrices.
Suppose the sequence $(\eps_k)_{k \geq1}$ consists of independent
Rademacher random
variables that are independent from the random matrices. Consider the
random series
\[
\sum_k \eps_k \mtx{W}_k.
\]
The summands may be strongly dependent on each other, but the
independence of the
Rademacher variables ensures that the summands are conditionally
independent and of zero
mean \eqref{eqn:cond-zero-mean} given $Z \mathrell=(\mtx
{W}_k)_{k \geq1}$.
\end{example}

\subsection{\texorpdfstring{A matrix Stein pair.}{A matrix Stein pair}}

Let us describe how to build a matrix Stein pair $(\mtx{X}, \mtx{X}')$
for the sum \eqref{eqn:cond-zero-mean-sum}
of conditionally independent, zero-mean random matrices.
The approach is similar to the case of an independent sum, which
appears in Section~\ref{sec:indep-sum}.
For each $k$, we draw a random matrix $\mtx{Y}_k'$ so that $\mtx{Y}_k'$
and $\mtx{Y}_k$ are conditionally i.i.d. given $(\mtx{Y}_j)_{j\neq k}$.
Then, independently, we draw an index $K$ uniformly at random from $\{
1, \ldots, n\}$.
As in Section~\ref{sec:indep-sum}, the random matrix
\[
\mtx{X}' \mathrell=\mtx{Y}_1 + \cdots+
\mtx{Y}_{K-1} + \mtx{Y}_K' + \mtx
{Y}_{K+1} + \cdots+ \mtx{Y}_n
\]
is an exchangeable counterpart to $\mtx{X}$. The conditional
independence and conditional zero-mean \eqref{eqn:cond-zero-mean} assumptions
imply that, almost surely,
\[
\Expect\bigl[ \mtx{Y}_k' \vertt(\mtx{Y}_j)_{j\neq k}
\bigr] = \Expect\bigl[ \mtx {Y}_k \vertt(\mtx{Y}_j)_{j\neq k}
\bigr] = \Expect\bigl[ \Expect[ \mtx{Y}_k \vertt Z ] \vertt(
\mtx{Y}_j)_{j\neq k} \bigr] = \mtx{0}. %
\]
Hence,
\begin{eqnarray*}
\Expect\bigl[\mtx{X} - \mtx{X}' \vertt(\mtx{Y}_j)_{j \geq1}
\bigr] &=& \Expect\bigl[ \mtx{Y}_K - \mtx{Y}_K'
\vertt(\mtx{Y}_j)_{j \geq1} \bigr]
\\
&=& \frac{1}{n} \sum_{k=1}^n \bigl(
\mtx{Y}_k - \Expect\bigl[ \mtx{Y}_k' \vertt(
\mtx{Y}_j)_{j\neq k} \bigr] \bigr) = \frac{1}{n} \sum
_{k=1}^n \mtx{Y}_k =
\frac{1}{n} \mtx{X}.
\end{eqnarray*}
Therefore, $(\mtx{X},\mtx{X}')$ is a matrix Stein pair with scale
factor $\alpha= n^{-1}$.

We can determine the conditional variance after a short argument that parallels
computation \eqref{eqn:indep-sum-DeltaX} in the independent setting,
%
\begin{eqnarray}
\label{eqn:cond-zero-mean-DeltaX} %
\mtxx{\Delta}_{\mtx{X}}
&=&
\frac{n}{2} \cdot\Expect \bigl[ \bigl(\mtx{Y}_K -
\mtx{Y}_K'\bigr)^2 \vertt (
\mtx{Y}_j)_{j \geq1} \bigr]
\nonumber
\\[-8pt]
\\[-8pt]
\nonumber
&=& \frac{1}{2} \sum_{k=1}^n \bigl(
\mtx{Y}_k^2 + \Expect\bigl[ \mtx {Y}_k^2
\vertt(\mtx{Y}_j)_{j \neq k} \bigr] \bigr).
\end{eqnarray}
Expression \eqref{eqn:cond-zero-mean-DeltaX} shows that, even in the
presence of some dependence,
we can control the size of the conditional expectation uniformly if we
control the size of the individual
summands.

Using the Stein pair $(\mtx{X}, \mtx{X}')$ and expression \eqref
{eqn:cond-zero-mean-DeltaX},
we may develop a variety of concentration inequalities for
conditionally independent, zero-mean sums that are
analogous to our results for independent sums. We omit detailed examples.

%

\section{\texorpdfstring{Combinatorial sums of matrices.}{Combinatorial sums of matrices}} \label{sec:combinatorial-sum}

The method of exchangeable pairs can also be applied to many types of
highly symmetric distributions.
In this section, we study a class of \emph{combinatorial matrix
statistics}, which generalize
the scalar statistics studied by Hoeffding \cite{Hoeffding51}.

\subsection{\texorpdfstring{Formulation.}{Formulation}}

Consider a deterministic array $( \mtx{A}_{jk} )_{j,k = 1}^n$ of
Hermitian matrices, and
let $\pi$ be a uniformly random permutation on $\{1, \ldots, n\}$.
Define the random matrix
%
\begin{equation}
\label{eqn:comb-sum} \mtx{Y} \mathrell=\sum_{j=1}^n
\mtx{A}_{j \pi(j)} \qquad\mbox{whose mean } \Expect\mtx{Y} = \frac{1}{n} \sum
_{j,k=1}^n \mtx{A}_{jk}.
\end{equation}
The combinatorial sum $\mtx{Y}$ is a natural candidate for an
exchangeable pair analysis.
Before we describe how to construct a matrix Stein pair, let us mention
a few problems that lead to a random matrix of the form $\mtx{Y}$.

%

\begin{example}[(Sampling without replacement)]
Consider a finite collection $\coll{B} \mathrell=\{ \mtx
{B}_1, \ldots,
\mtx
{B}_n \}$ of deterministic Hermitian matrices.
Suppose that we want to study a sum of $s$ matrices sampled randomly
from $\coll{B}$ without replacement.
We can express this type of series in the form
\[
\mtx{W} \mathrell=\sum_{j=1}^s
\mtx{B}_{\pi(j)},
\]
where $\pi$ is a random permutation on $\{1, \ldots, n\}$. The matrix
$\mtx{W}$
is therefore an example of a combinatorial sum.
\end{example}

\begin{example}[(A randomized ``inner product'')]
Consider two fixed sequences of complex matrices
\[
\mtx{B}_1, \ldots, \mtx{B}_n \in\C^{d_1 \times s}
\quad\mbox{and}\quad \mtx{C}_1, \ldots, \mtx{C}_n \in
\C^{s \times d_2}.
\]
We may form a permuted matrix ``inner product'' by
arranging one sequence in random order, multiplying
the elements of the two sequences together, and summing the terms. That is,
we are interested in the random matrix
\[
\mtx{Z} \mathrell=\sum_{j=1}^n
\mtx{B}_j \mtx{C}_{\pi(j)}.
\]
This random matrix $\coll{D}( \mtx{Z} )$ is a combinatorial
sum of Hermitian matrices.
\end{example}

\subsection{\texorpdfstring{A matrix Stein pair.}{A matrix Stein pair}} \label{sec:diag-perm}

To study the combinatorial sum \eqref{eqn:comb-sum} of matrices using
the method of exchangeable pairs,
we first introduce the zero-mean random matrix
\[
\mtx{X} \mathrell=\mtx{Y} - \Expect\mtx{Y}.
\]
To construct a matrix Stein pair $(\mtx{X}, \mtx{X}')$,
we draw a pair $(J, K)$ of indices independently of $\pi$ and uniformly
at random from $\{1, \ldots, n\}^2$.
Define a second random permutation $\pi' \mathrell=\pi
\circ(J, K)$ by
composing $\pi$ with the transposition of the random indices $J$ and
$K$. The pair $(\pi, \pi')$ is exchangeable, so
\[
\mtx{X}' \mathrell=\sum_{j=1}^n
\mtx{A}_{j \pi'(j)} - \Expect\mtx{Y}
\]
is an exchangeable counterpart to $\mtx{X}$.

To verify that $(\mtx{X}, \mtx{X}')$ is a matrix Stein pair, we
calculate that
\begin{eqnarray*}
\Expect\bigl[ \mtx{X} - \mtx{X}' \vertt\pi\bigr] &=&
\Expect [ \mtx{A}_{J\pi(J)} + \mtx{A}_{K\pi(K)} - \mtx
{A}_{J\pi
(K)} - \mtx{A}_{K \pi(J)} \vertt\pi ]
\\
&=& \frac{1}{n^2} \sum_{j, k = 1}^n [
\mtx{A}_{j\pi
(j)} + \mtx{A}_{k\pi(k)} - \mtx{A}_{j \pi(k)} -
\mtx{A}_{k \pi(j)} ]
\\
&=& \frac{2}{n} (\mtx{Y} - \Expect\mtx{Y}) =
\frac{2}{n} \mtx{X}.
\end{eqnarray*}
The first identity holds because the %
sums $\mtx{X}$ and $\mtx{X}'$ differ for only
four choices of indices. Thus %
$(\mtx{X}, \mtx{X}')$ is a Stein pair with scale factor $\alpha= 2/n$.

Turning to the conditional variance, we find that
%
\begin{eqnarray}
\label{eqn:diag-perm-DeltaX} \mtxx{\Delta}_{\mtx{X}}(\pi)
\nonumber
&=& \frac{n}{4}
\Expect \bigl[ \bigl(\mtx{X} - \mtx{X}'\bigr)^2 \vertt\pi
\bigr]
\nonumber
\\[-8pt]
\\[-8pt]
\nonumber
&=& \frac{1}{4n} \sum_{j,k=1}^n [
\mtx{A}_{j\pi(j)} + \mtx {A}_{k \pi(k)} - \mtx{A}_{j \pi(k)} -
\mtx{A}_{k \pi(j)} ]^2.
\end{eqnarray}
%
The structure of the conditional variance differs from previous
examples, but
we recognize that $\mtxx{\Delta}_{\mtx{X}}$ is controlled when the
matrices $\mtx{A}_{jk}$
are bounded. %
%

%

%

%

\subsection{\texorpdfstring{Exponential concentration for a
combinatorial sum.}
{Exponential concentration for a combinatorial sum}} \label
{sec:diag-perm-2}

We can apply our matrix concentration results to study the behavior of
a combinatorial
sum of matrices. As an example, let us present a Bernstein-type inequality.
The argument is similar to the proof of Corollary~\ref{cor:bernstein},
so we leave the details to Appendix \ref{sec:pf-diag-perm-2}.

%
\begin{cor}[(Bernstein inequality for a combinatorial matrix sum)] \label
{cor:diag-perm-2}
Consider an array $( \mtx{A}_{jk} )_{j, k = 1}^n$ of deterministic
matrices in $\mathbb{H}^{d}$ that satisfy
\[
\sum_{j,k=1}^n \mtx{A}_{jk} =
\mtx{0} %
\quad\mbox{and}\quad \llVert {\mtx{A}_{jk} } \rrVert \leq R\qquad
\mbox{for each pair $(j, k)$ of indices.}
\]
Define the random matrix
$
\mtx{X} \mathrell=\sum_{j=1}^n \mtx{A}_{j \pi(j)}
$,
where $\pi$ is a uniformly random permutation on $\{ 1, \ldots, n \}$.
Then, for all $t \geq0$,
\[
\mathbb{P} \bigl\{ {\lambda_{\max}(\mtx{X}) \geq t } \bigr\} \leq d
\cdot\exp \biggl\{ \frac{-t^2}{12 \sigma^2 + 4\sqrt{2} R t} \biggr\} \qquad\mbox{for } \sigma^2
\mathrell=\frac{1}{n} \Biggl\llVert {\sum
_{j,k=1}^n \mtx {A}_{jk}^2 }
\Biggr\rrVert.
\]
Furthermore,
\[
\Expect\lambda_{\max}(\mtx{X}) \leq\sigma\sqrt{12 \log d} + 2\sqrt{2} R
\log d.
\]
\end{cor}

%

\section{\texorpdfstring{Self-reproducing matrix functions.}{Self-reproducing matrix functions}} \label{sec:self-repro}

The method of exchangeable pairs can also be used to analyze nonlinear
matrix-valued functions of random variables.
In this section, we explain how to analyze matrix functions that
satisfy a \emph{self-reproducing property}.

\subsection{\texorpdfstring{Example: Matrix second-order Rademacher chaos.}
{Example: Matrix second-order Rademacher chaos}}

We begin with an example that shows how the self-reproducing property
might arise.
Consider a quadratic form that takes on random matrix values
%
\begin{equation}
\label{eqn:second-rad-chaos} \mtx{H}(\vct{\eps}) \mathrell=\sum
_k \sum_{j < k} \eps
_j \eps_k \mtx{A}_{jk}.
\end{equation}
In this expression, $\vct{\eps}$ is a finite vector of independent Rademacher
random variables. The array $( \mtx{A}_{jk} )_{j,k \geq1}$ consists of
deterministic Hermitian matrices,
and we assume that $\mtx{A}_{jk} = \mtx{A}_{kj}$.
Observe that the summands in $\mtx{H}(\vct{\eps})$ are dependent, and
they do not satisfy the conditional zero-mean property \eqref
{eqn:cond-zero-mean} in general.
Nevertheless, $\mtx{H}(\vct{\eps})$ does satisfy a fruitful
self-reproducing property
\begin{eqnarray*}
\sum_k \bigl(\mtx{H}(\vct{\eps}) - \Expect\bigl[
\mtx{H}(\vct {\eps}) \vertt(\eps_j)_{j\neq k} \bigr]\bigr) &=&
\sum_k \sum_{j \neq k}
\eps_j \bigl(\eps_k - \Expect[ \eps_k ]
\bigr)\mtx{A}_{jk}
\\
&=& \sum_k \sum_{j \neq k}
\eps_j \eps_k \mtx{A}_{jk} = 2 \mtx{H}(\vct{
\eps}).
\end{eqnarray*}
We have applied the pull-through property of conditional expectation,
the assumption that the Rademacher variables are independent and the
fact that $\mtx{A}_{jk} = \mtx{A}_{kj}$. As we will see, this type of
self-reproducing condition can be used to construct a matrix Stein pair.

A random matrix of the form \eqref{eqn:second-rad-chaos}
is called a \emph{second-order Rademacher chaos}.
This class of random matrices arises in a variety of situations, including
randomized linear algebra \cite{CD12Matrix-Probing},
compressed sensing \cite{Rau10Compressive-Sensing}, Section~9,
and chance-constrained optimization \cite{CheungSoWa11}.
Indeed, concentration inequalities for the matrix-valued
Rademacher chaos have many potential applications.

\subsection{\texorpdfstring{Formulation and matrix Stein pair.}
{Formulation and matrix Stein pair}}

In this section, we describe a more general version of the
self-reproducing property.
Suppose that $\vctt{z} \mathrell=(Z_1, \ldots, Z_n)$ is a
random vector
taking values in a
Polish space $\metricspace$. First, we construct an exchangeable counterpart
%
\begin{equation}
\label{eqn:zee-prime} \vctt{z}' \mathrell=\bigl(Z_1,
\ldots, Z_{K-1}, Z_K', Z_{K+1}, \ldots,
Z_n\bigr),
\end{equation}
where $Z_k$ and $Z_k'$ are conditionally i.i.d. given $(Z_j)_{j \neq
k}$, and $K$ is an independent coordinate
drawn uniformly at random from $\{1, \ldots, n\}$.

Next, let $\mtx{H} \dvtx\metricspace\to\mathbb{H}^{d}$ be a bounded
measurable function.
Assume that $\mtx{H}(\vctt{z})$ satisfies an abstract \emph
{self-reproducing property}:
for a parameter $s > 0$,
\[
\sum_{k=1}^n
\bigl(\mtx{H}(\vctt{z}) - \Expect\bigl[ \mtx {H}(\vctt {z}) \vertt(Z_j)_{j \neq k}
\bigr] \bigr) = s \cdot\bigl(\mtx{H}(\vctt{z}) - \Expect\mtx{H}(\vctt{z})\bigr)
\qquad\mbox {almost surely.}
\]
Under this assumption, we can easily check that the random matrices
\[
\mtx{X} \mathrell=\mtx{H}(\vctt{z}) - \Expect\mtx {H}(\vctt{z})
\quad\mbox{and}\quad \mtx{X}' \mathrell=\mtx{H}\bigl(
\vctt{z}'\bigr) - \Expect\mtx {H}(\vctt{z})
\]
form a matrix Stein pair. Indeed,
\[
\Expect\bigl[ \mtx{X} - \mtx{X}' \vertt\vctt{z} \bigr] = \Expect
\bigl[ \mtx{H}(\vctt{z}) - \mtx{H}\bigl(\vctt{z}'\bigr) \vertt\vctt{z}
\bigr] = \frac{s}{n} \bigl( \mtx{H}(\vctt{z}) - \Expect\mtx{H}(\vctt{z})
\bigr) = \frac{s}{n} \mtx{X}.
\]
We see that $(\mtx{X},\mtx{X}')$ is a matrix Stein pair with scaling
factor $\alpha= s/n$.

Finally, we compute the conditional variance
%
\begin{eqnarray}\label{eqn:self-repro-DeltaX}
\mtxx{\Delta}_{\mtx{X}}(\vctt{z}) &=& \frac{n}{2s} \Expect
\bigl[ \bigl(\mtx {H}(\vctt{z}) - \mtx{H}\bigl(\vctt{z}'\bigr)
\bigr)^2 \vertt\vctt{z} \bigr]
\nonumber
\\[-8pt]
\\[-8pt]
\nonumber
 &= &\frac{1}{2s} \sum_{k=1}^n
\Expect \bigl[ \bigl(\mtx{H}(\vctt{z}) - \mtx{H}\bigl(Z_1, \ldots,
Z_k', \ldots, Z_n\bigr) \bigr)^2
\vertt\vctt{z} \bigr].
\end{eqnarray}
We discover that the conditional variance is small when $\mtx{H}$ has
controlled coordinate differences.
In this case, the method of exchangeable pairs provides good concentration
inequalities for the random matrix $\mtx{X}$.

%

%

\subsection{\texorpdfstring{Matrix bounded differences inequality.}
{Matrix bounded differences inequality}}

As an example, %
we can develop a bounded differences inequality
for random matrices by appealing to Theorem~\ref{thmm:concentration-bdd}.

\begin{cor}[(Matrix bounded differences)] \label{cor:bound-diff}
Let $\vctt{z} \mathrell=(Z_1, \ldots, Z_n)$ be a random
vector taking
values in a Polish space $\metricspace$,
and, for each index $k$, let $Z_k'$ and $Z_k$ be conditionally
i.i.d. given $(Z_j)_{j \neq k}$.
Suppose that $\mtx{H}\dvtx\metricspace\to\mathbb{H}^{d}$ is a
function that satisfies
the self-reproducing property
\[
\sum_{k=1}^n \bigl(\mtx{H}(\vctt{z}) -
\Expect\bigl[ \mtx {H}(\vctt {z}) \vertt(Z_j)_{j \neq k} \bigr]
\bigr) = s \cdot\bigl(\mtx{H}(\vctt{z}) - \Expect\mtx{H}(\vctt{z})\bigr)\qquad \mbox
{almost surely}
\]
for a parameter $s > 0$ as well as the bounded differences condition
%
\begin{eqnarray}
\label{eqn:bdd-diff}\qquad
&\Expect \bigl[ \bigl(\mtx{H}(\vctt{z}) - \mtx{H}
\bigl(Z_1, \ldots, Z_k', \ldots,
Z_n\bigr) \bigr)^2 \vertt\vctt{z} \bigr] \psdle
\mtx{A}_k^2\qquad \mbox{for each index $k$}
\end{eqnarray}
almost surely, where $\mtx{A}_k$ is a deterministic matrix in $\mathbb
{H}^{d}$.
Then, for all $t\geq0$,
\[
\mathbb{P} \bigl\{ {\lambda_{\max}\bigl(\mtx{H}(\vctt{z}) - \Expect\mtx
{H}(\vctt {z})\bigr) \geq t} \bigr\} \leq d \cdot\econst^{-st^2/L} \qquad\mbox{for }
L \mathrell=\Biggl\llVert {\sum_{k=1}^n
\mtx{A}_k^2 } \Biggr\rrVert.
\]
Furthermore,
\[
\Expect\lambda_{\max}\bigl(\mtx{H}(\vctt{z}) - \Expect\mtx{H}(\vctt{z})
\bigr) \leq\sqrt{\frac{L \log d}{s}}.
\]
\end{cor}

In the scalar setting, Corollary~\ref{cor:bound-diff} reduces to a
version of
McDiarmid's bounded difference inequality \cite{McDiarmid89}. The
result also complements the
matrix bounded difference inequality of \cite{Tro11User-Friendly-FOCM},
Corollary~7.5, which
requires independent input variables but makes no self-reproducing assumption.

\begin{pf*}{Proof of Corollary \ref{cor:bound-diff}}
Since $\mtx{H}(\vctt{z})$ is self-reproducing,
we may construct a matrix Stein pair $(\mtx{X}, \mtx{X}')$
with scale factor $\alpha= s/n$ as in Section~\ref{sec:self-repro}.
According to \eqref{eqn:self-repro-DeltaX}, the conditional variance
of the pair satisfies
\begin{eqnarray*}
\mtxx{\Delta}_{\mtx{X}} &= &\frac{1}{2s} \sum
_{k=1}^n \Expect \bigl[ \bigl(\mtx{H}(\vctt{z}) -
\mtx{H}\bigl(Z_1, \ldots, Z_k', \ldots,
Z_n\bigr)\bigr)^2 \vertt\vctt{z} \bigr]
\\
&\psdle&\frac{1}{2s} \sum_{k=1}^n
\mtx{A}_k^2 \psdle\frac
{L}{2s}\cdot\Id.
\end{eqnarray*}
We have used the bounded differences condition \eqref{eqn:bdd-diff} and
the definition of the bound $L$.
To complete the proof, we apply the concentration result, Theorem~\ref
{thmm:concentration-bdd},
with the parameters $c = 0$ and $v = L/2s$.
\end{pf*}

\begin{appendix}\label{app}
\section{\texorpdfstring{Proof of Theorem \lowercase{\protect\ref{thmm:concentration-subgauss}}}
{Proof of Theorem 5.1}} \label{sec:subgauss-pf}
The proof of the refined exponential concentration bound, Theorem~\ref
{thmm:concentration-subgauss}, parallels the argument in Theorem~\ref
{thmm:concentration-bdd}, but it differs at an important point. In the
earlier result, we used an almost sure bound on the conditional
variance to control the derivative of the trace m.g.f. This time, we use
entropy inequalities to introduce finer information about the behavior
of the conditional variance. The proof is essentially a matrix version
of Chatterjee's argument \cite{Cha08Concentration-Inequalities}, Theorem~3.13.

Our main object is to bound the trace m.g.f. of $\mtx{X}$ in terms of the
trace m.g.f. of the conditional variance.
The next result summarizes our bounds.
%
\begin{lemma} [(Refined trace m.g.f. estimates)] \label{lem:mgf-control}
Let $(\mtx{X}, \mtx{X}')$ be a matrix Stein pair, and assume that
$\mtx
{X}$ is almost surely bounded in norm.
Then the normalized trace m.g.f. $m(\theta) \mathrell=\Expect
\ntr\econst
^{\theta\mtx{X}}$ satisfies the bounds
%
\begin{eqnarray}\label{eqn:mgf-control}
\log m(\theta) &\leq& \frac{1}{2} \log \biggl(
\frac{1}{1-\theta^2/\psi} \biggr) \log \Expect\ntr\econst^{\psi\mtxx{\Delta}_{\mtx{X}}}
\nonumber
\\[-8pt]
\\[-8pt]
\nonumber
&\leq& \frac{ \theta^2/\psi}{2 (1 - \theta^2/\psi)} \log \Expect \ntr
\econst^{\psi\mtxx{\Delta}_{\mtx{X}}} %
\qquad\mbox{for $\psi> 0$ and $0 \leq\theta< \sqrt{
\psi}$.}
\end{eqnarray}
\end{lemma}
We establish Lemma~\ref{lem:mgf-control} in Section~\ref{sec:pf-claim}
et seq.
Afterward, in Section~\ref{sec:fancy-lt-arg}, we invoke the matrix
Laplace transform bound to complete the proof of Theorem~\ref
{thmm:concentration-subgauss}.

\subsection{\texorpdfstring{The derivative of the trace m.g.f.}
{The derivative of the trace m.g.f.}} \label{sec:pf-claim}

The first steps of the argument are the same as in the proof of
Theorem~\ref{thmm:concentration-bdd}.
Since $\mtx{X}$ is almost surely bounded, we need not worry about
regularity conditions.
The derivative of the trace m.g.f. satisfies
%
\begin{equation}
\label{eqn:tmgf-prime-2} m'(\theta) = \Expect\trace \bigl[ \mtx{X}
\econst^{\theta\mtx{X}} \bigr]\qquad \mbox{for $\theta\in\R$.}
\end{equation}
Lemma~\ref{lem:mgf-derivative} provides a bound for the derivative in
terms of
the conditional variance,
%
\begin{equation}
\label{eqn:m-prime-2} m'(\theta) \leq\theta\cdot \Expect\ntr \bigl[ \mtx{
\Delta}_{\mtx{X}} \econst^{\theta\mtx{X}} \bigr] \qquad\mbox{for $\theta\geq0$.}
\end{equation}
In the proof of Lemma~\ref{lem:mgf-bounds}, we applied an almost sure
bound for the conditional variance to control the derivative of the
m.g.f. This time, we incorporate information about the typical size of
$\mtxx{\Delta}_{\mtx{X}}$ by developing a bound in terms of the function
$r(\psi)$.

\subsection{\texorpdfstring{Entropy for random matrices and duality.}
{Entropy for random matrices and duality}} \label{sec:entropy}

Let us introduce an entropy function for random matrices.
%

\begin{defn}[(Entropy for random matrices)] \label{def:matrix-entropy}
Let $\mtx{W}$ be a random matrix in $\mathbb{H}_{+}^{d}$ subject to the
normalization $\Expect\ntr\mtx{W} = 1$.
The \emph{\textup{(}negative\textup{)} matrix entropy} is defined as
%
\begin{equation}
\label{eqn:entropy} \ent(\mtx{W}) \mathrell=\Expect\ntr( \mtx{W} \log
\mtx{W}). %
\end{equation}
We enforce the convention that $0 \log0 = 0$.
\end{defn}

The matrix entropy is relevant to our discussion because its
Fenchel--Legendre conjugate is the cumulant generating function.
The Young inequality for matrix entropy offers one way to formulate this
duality relationship.

\begin{prop}[(Young inequality for matrix entropy)] \label{prop:entropy-duality}
Suppose that $\mtx{V}$ is a random matrix in $\mathbb{H}^{d}$ that is
almost surely
bounded in norm, and suppose that $\mtx{W}$ is a random matrix in
$\mathbb{H}_{+}^{d}$ subject to the
normalization $\Expect\ntr\mtx{W} = 1$.
Then
\[
\Expect\ntr( \mtx{VW} ) \leq\log\Expect\ntr\econst^{\mtx{V}} + \ent(
\mtx{W}).
\]
\end{prop}

Proposition~\ref{prop:entropy-duality} follows from a variant of the
argument in \cite{Car10Trace-Inequalities}, Theorem~2.13.

%

\subsection{\texorpdfstring{A refined differential inequality for
the trace m.g.f.}
{A refined differential inequality for the trace m.g.f.}}

We intend to apply the Young inequality for matrix entropy to decouple
the product of random matrices in \eqref{eqn:m-prime-2}.
First, we must rescale the exponential in \eqref{eqn:m-prime-2}, so its
expected trace equals one,
%
\begin{equation}
\label{eqn:W-theta} \mtx{W}(\theta) \mathrell=\frac{1}{\Expect\ntr\econst
^{\theta\mtx{X}}} \cdot
\econst^{\theta\mtx{X}} = \frac{1}{m(\theta)} \cdot\econst^{\theta\mtx{X}}.
\end{equation}
For each $\psi> 0$, we can rewrite \eqref{eqn:m-prime-2} as
\[
m'(\theta) \leq\frac{\theta m(\theta)}{\psi} \cdot\Expect\ntr \bigl[ \psi\mtx{
\Delta}_{\mtx{X}} \cdot \mtx{W}(\theta) \bigr]. %
\]
The Young inequality for matrix entropy, Proposition~\ref
{prop:entropy-duality}, implies that
%
\begin{equation}
\label{eqn:m-prime-entropy} m'(\theta) \leq\frac{\theta m(\theta)}{\psi} \bigl[ \log
\Expect\ntr\econst^{\psi\mtxx{\Delta}_{\mtx{X}}} + \ent\bigl( \mtx{W}(\theta) \bigr) \bigr].
\end{equation}
The first term in the bracket is precisely $\psi r(\psi)$.
Let us examine the second term more closely.

To control the matrix entropy of $\mtx{W}(\theta)$, we need to bound
its logarithm. Referring back to definition \eqref{eqn:W-theta}, we
see that
%
\begin{equation}
\label{eqn:log-W-theta} \log\mtx{W}(\theta) = \theta\mtx{X} - \bigl( \log\Expect\ntr
\econst ^{\theta\mtx{X}} \bigr) \cdot\Id \psdle\theta\mtx{X} - \bigl( \log\ntr
\econst^{\theta\Expect\mtx
{X}} \bigr) \cdot\Id = \theta\mtx{X}. %
\end{equation}
The second relation depends on Jensen's inequality and the fact that
the trace exponential is convex \cite{Pet94Survey-Certain}, Section~2.
The third relation relies on the property that $\Expect\mtx{X} = \mtx{0}$.
Since the matrix $\mtx{W}(\theta)$ is positive, %
we can substitute the semidefinite bound \eqref{eqn:log-W-theta} into the
definition \eqref{eqn:entropy} of the matrix entropy,
\begin{eqnarray*}
\ent\bigl(\mtx{W}(\theta)\bigr) &=& \Expect\ntr \bigl[ \mtx{W}(
\theta) \cdot\log\mtx{W}(\theta) \bigr]
\\
&\leq&\theta\cdot\Expect\ntr \bigl[ \mtx{W}(\theta) \cdot\mtx{X}
\bigr] %
= \frac{\theta}{m(\theta)} \cdot\Expect\ntr \bigl[ \mtx{X} \econst
^{\theta\mtx{X}} \bigr]. %
\end{eqnarray*}
We have reintroduced the definition \eqref{eqn:W-theta} of $\mtx
{W}(\theta)$ in the last relation.
Identify the derivative \eqref{eqn:tmgf-prime-2} of the trace m.g.f.
to reach
%
\begin{equation}
\label{eqn:m-prime-part1} \ent\bigl(\mtx{W}(\theta)\bigr) \leq\frac{\theta m'(\theta)}{m(\theta)}.
\end{equation}

To establish a differential inequality, substitute the
definition \eqref
{eqn:r-psi} of $r(\psi)$
and the bound \eqref{eqn:m-prime-part1} into the estimate \eqref
{eqn:m-prime-entropy} to discover that
\[
m'(\theta) \leq\frac{\theta m(\theta)}{\psi} \biggl[ \psi r(\psi) +
\frac{\theta m'(\theta)}{m(\theta)} \biggr] = r(\psi) \theta\cdot m(\theta) + \frac{\theta^2}{\psi}
\cdot m'(\theta).
\]
Rearrange this formula to isolate the log-derivative $m'(\theta
)/m(\theta)$ of the trace m.g.f. We conclude that
%
\begin{equation}
\label{eqn:rpsi-diff-ineq} \frac{\mathrm{d}}{\mathrm{d}{\theta}} \log m(\theta) \leq\frac{
r(\psi) \theta}{1 -
\theta
^2 / \psi}\qquad
\mbox{for $0 \leq\theta< \sqrt{\psi}$.}
\end{equation}

\subsection{\texorpdfstring{Solving the differential inequality.}
{Solving the differential inequality}}

To integrate \eqref{eqn:rpsi-diff-ineq}, recall that $\log m(0) = 0$,
and invoke
the fundamental theorem of calculus %
to reach
\[
\log m(\theta) = \int_0^\theta\frac{\mathrm{d}}{\mathrm{d}{s}}
\log m(s) \,\mathrm{d} {s} \leq\int_0^\theta
\frac{r(\psi) s}{1 - s^2/\psi} \,\mathrm{d} {s} = \frac{\psi r(\psi) }{2} \log \biggl(
\frac{1}{1-\theta^2/\psi
} \biggr).
\]
We can develop a weaker inequality by making a further approximation
within the integral
\[
\log m(\theta) \leq\int_0^\theta\frac{r(\psi) s}{1 - s^2/\psi}
\,\mathrm{d} {s} \leq\int_0^\theta\frac{r(\psi) s}{1 - \theta^2/\psi}
\,\mathrm{d} {s} = \frac{ r(\psi) \theta^2 }{2(1 - \theta^2/\psi)}.
\]
These calculations are valid when $0 \leq\theta< \sqrt{\psi}$, so
claim \eqref{eqn:mgf-control} follows.

%

%

\subsection{\texorpdfstring{The matrix Laplace transform argument.}
{The matrix Laplace transform argument}} \label{sec:fancy-lt-arg}

With the trace m.g.f. bound~\eqref{eqn:mgf-control} at hand,
we can complete the proof of Theorem~\ref{thmm:concentration-subgauss}.
Proposition~\ref{prop:matrix-laplace}, the matrix Laplace transform
method, yields the estimate
\begin{eqnarray*}
\mathbb{P} \bigl\{ {\lambda_{\max} (\mtx{X}) \geq t} \bigr\}
&\leq& d \cdot\inf_{0 < \theta< \sqrt{\psi}}  \exp \biggl\{ -\theta t +
\frac{r(\psi) \theta^2}{2(1 - \theta^2/\psi)} \biggr\}
\\
&\leq &d \cdot\inf_{0 < \theta< \sqrt{\psi}}  \exp \biggl\{ -\theta t +
\frac{r(\psi) \theta^2}{2(1 - \theta/\sqrt{\psi})} \biggr\}
\\
&=& d \cdot\exp \biggl\{ - \frac{r(\psi)\psi}{2} \bigl(1-\sqrt {1+2t/\bigl(r(\psi)
\sqrt{\psi}\bigr)} \bigr)^2 \biggr\}
\\
&\leq& d \cdot\exp \biggl\{ - \frac{t^2}{2r(\psi) + 2t/\sqrt{\psi}} \biggr\},
\end{eqnarray*}
since the infimum occurs at $\theta= \sqrt{\psi}-\sqrt{\psi}/\sqrt {1+2t/(r(\psi)\sqrt{\psi}})$.
This delivers the tail bound \eqref{eqn:refined-tail}.

To establish inequality \eqref{eqn:refined-mean} for the expectation of
the maximum eigenvalue, we can apply Proposition~\ref
{prop:matrix-laplace} and the trace m.g.f. bound \eqref{eqn:mgf-control} a
second time. Indeed,
\begin{eqnarray*}
\Expect\lambda_{\max}(\mtx{X}) &\leq&\inf_{0 < \theta< \sqrt{\psi}}  \frac{1}{\theta} \biggl[ \log d + \frac{r(\psi) \theta^2}{2(1-\theta^2/\psi)} \biggr]
\\
&\leq&\inf_{0 < \theta< \sqrt{\psi}}  \frac{1}{\theta} \biggl[ \log d +
\frac{r(\psi) \theta^2}{2(1-\theta/\sqrt{\psi})} \biggr] = \sqrt{2 r(\psi) \log d} + \frac{\log d}{\sqrt{\psi}}.
\end{eqnarray*}
This completes the proof of Theorem~\ref{thmm:concentration-subgauss}.

\section{\texorpdfstring{Proof of Theorem \lowercase{\protect\ref{cor:ros-pin}}}{Proof of Theorem 7.4}}
\label{sec:pf-rosenthal}

The proof of the matrix Rosenthal inequality takes place in two steps.
First, we verify that the bound \eqref{eqn:rosenthal-pos} holds for psd
random matrices.
Then, we use this result to provide a short proof of the bound \eqref
{eqn:rosenthal}
for Hermitian random matrices. Before we start, let us remind the
reader that the
\emph{$L_p$ norm} of a scalar random variable $Z$ is given by %
$(\Expect\llvert {Z} \rrvert ^p )^{1/p}$ for each $p \geq1$.

\subsection{\texorpdfstring{A sum of random psd matrices.}{A sum of random psd matrices}}

We begin with the moment bound \eqref{eqn:rosenthal-pos} for an
independent sum of random psd matrices.
Introduce the quantity of interest
\[
E^2 \mathrell= \biggl( \Expect\biggl\llVert {\sum
_k \mtx {P}_k } \biggr\rrVert
_{2p}^{2p} \biggr)^{1/(2p)}.
\]
We may invoke the triangle inequality for the $L_{2p}$ norm to obtain
\begin{eqnarray*}
E^2 &\leq& \biggl( \Expect\biggl\llVert {\sum
_k (\mtx{P}_k - \Expect \mtx{P}_k)
} \biggr\rrVert _{2p}^{2p} \biggr)^{1/(2p)} + \biggl
\llVert {\sum_k \Expect\mtx{P}_k }
\biggr\rrVert _{2p}
\\
&=&\dvtx \bigl( \Expect\llVert {\mtx{X} } \rrVert _{2p}^{2p}
\bigr)^{1/(2p)} + \mu.
\end{eqnarray*}
We can apply the matrix BDG inequality to control this expectation,
which yields
an algebraic inequality between $E^2$ and $E$. We solve this inequality
to bound $E^2$.

The series $\mtx{X}$ consists of centered, independent random matrices,
so we can use the Stein pair described in Section~\ref{sec:indep-sum}.
According to \eqref{eqn:indep-sum-DeltaX}, the conditional variance
$\mtxx{\Delta}_{\mtx{X}}$ takes the form
\begin{eqnarray*}
\mtxx{\Delta}_{\mtx{X}} %
&=& \frac{1}{2} \sum
_k \bigl[ (\mtx{P}_k - \Expect
\mtx{P}_k)^2 + \Expect( \mtx{P}_k - \Expect
\mtx{P}_k)^2 \bigr]
\\
&\psdle&\frac{1}{2} \sum_k \bigl[ 2
\mtx{P}_k^2 + 2 (\Expect \mtx{P}_k)^2
+ \Expect\mtx{P}_k^2 - (\Expect\mtx{P}_k)^2
\bigr]
\\
&\psdle&\sum_k \bigl(\mtx{P}_k^2
+ \Expect\mtx{P}_k^2 \bigr).
\end{eqnarray*}
The first inequality follows from the operator convexity \eqref
{eqn:square-convex} of the square function;
the second expectation is computed exactly. The last bound uses the
operator Jensen inequality \eqref{eqn:kadison}.
Now, the matrix BDG inequality yields
\begin{eqnarray*}
E^2 &\leq&\sqrt{2p-1} \cdot \bigl( \Expect\llVert {\mtxx{\Delta
}_{\mtx {X}} } \rrVert _{p}^{p} \bigr)^{1/(2p)}
+ \mu
\\
&\leq&\sqrt{2p-1} \cdot \biggl( \Expect\biggl\llVert {\sum
_k \bigl(\mtx{P}_k^2 + \Expect
\mtx{P}_k^2 \bigr) } \biggr\rrVert _{p}^p
\biggr)^{1/(2p)} + \mu
\\
&\leq&\sqrt{4p-2} \cdot \biggl( \Expect\biggl\llVert {\sum
_k \mtx {P}_k^2 } \biggr\rrVert
_{p}^p \biggr)^{1/(2p)} + \mu.
\end{eqnarray*}
The third line follows from the triangle inequality for the $L_p$ norm
and Jensen's inequality.

Next, we search for a copy of $E^2$ inside this expectation.
To accomplish this goal, we want to draw a factor $\mtx{P}_k$ off of
each term in
the sum. The following result of Pisier and Xu \cite{PX97Noncommutative-Martingale},
Lemma~2.6,
has the form we desire.

\begin{prop}[(A matrix Schwarz-type inequality)] \label{prop:fancy-holder}
Consider a finite sequence $(\mtx{A}_k)_{k \geq1}$ of deterministic
psd matrices. For each $p \geq1$,
\[
\biggl\llVert {\sum_k \mtx{A}_k^2
} \biggr\rrVert _{p} \leq \biggl( \sum_k
\llVert {\mtx{A}_k } \rrVert _{2p}^{2p}
\biggr)^{1/(2p)} \biggl\llVert {\sum_k
\mtx{A}_k } \biggr\rrVert _{2p}.
\]
\end{prop}

Apply the matrix Schwarz-type inequality, Proposition~\ref
{prop:fancy-holder}, to reach
\begin{eqnarray*}
E^2 &\leq&\sqrt{4p-2} \cdot \biggl[ \Expect \biggl( \sum
_k \llVert {\mtx{P}_k} \rrVert
_{2p}^{2p} \biggr)^{1/2} \biggl\llVert {\sum
_k \mtx{P}_k } \biggr\rrVert
_{2p}^{p} \biggr]^{1/(2p)} + \mu
\\
&\leq&\sqrt{4p-2} \cdot \biggl( \sum_k \Expect
\llVert {\mtx {P}_k} \rrVert _{2p}^{2p}
\biggr)^{1/(4p)} \biggl( \Expect\biggl\llVert {\sum
_k \mtx{P}_k } \biggr\rrVert
_{2p}^{2p} \biggr)^{1/(4p)} + \mu. %
\end{eqnarray*}
The second bound is the Cauchy--Schwarz inequality for expectation.
The resulting estimate takes the form $E^2 \leq c E + \mu$.
Solutions of this quadratic
inequality must satisfy $E \leq c + \sqrt{\mu}$.
We reach
\[
E \leq\sqrt{4p-2} \cdot \biggl( \sum_k \Expect
\llVert {\mtx {P}_k} \rrVert _{2p}^{2p}
\biggr)^{1/(4p)} + \biggl\llVert {\sum_k
\Expect\mtx{P}_k } \biggr\rrVert _{2p}^{1/2}.
\]
Square this expression to complete the proof of \eqref{eqn:rosenthal-pos}.

%

%

%

\subsection{\texorpdfstring{A sum of centered, random Hermitian matrices.}
{A sum of centered, random Hermitian matrices}}

We are now prepared to establish bound \eqref{eqn:rosenthal}
for a sum of centered, independent, random Hermitian matrices.
Define the random matrix $\mtx{X} \mathrell=\sum_k \mtx{Y}_k$.
We may use the matrix Stein pair described in Section~\ref{sec:indep-sum}.
According to \eqref{eqn:indep-sum-DeltaX},
the conditional variance $\mtxx{\Delta}_{\mtx{X}}$ takes the form
\[
\mtxx{\Delta}_{\mtx{X}} = \frac{1}{2} \sum
_k \bigl(\mtx {Y}_k^2 + \Expect
\mtx{Y}_k^2 \bigr).
\]
The matrix BDG inequality, Theorem~\ref{thmm:BDG-inequality}, yields
\begin{eqnarray*}
\bigl( \Expect\llVert {\mtx{X}} \rrVert _{4p}^{4p}
\bigr)^{1/(4p)} &\leq&\sqrt{4p-1} \cdot \bigl( \Expect\llVert {\mtxx{\Delta
}_{\mtx {X}} } \rrVert _{2p}^{2p} \bigr)^{1/(4p)}
\\
&=& \sqrt{4p-1} \cdot \biggl( \Expect\biggl\llVert {\frac{1}{2} \sum
_k \bigl(\mtx{Y}_k^2 +
\Expect\mtx{Y}_k^2 \bigr) } \biggr\rrVert
_{2p}^{2p} \biggr)^{1/(4p)}
\\
&\leq&\sqrt{4p-1} \cdot \biggl( \Expect\biggl\llVert {\sum
_k \mtx {Y}_k^2 } \biggr\rrVert
_{2p}^{2p} \biggr)^{1/(4p)}.
\end{eqnarray*}
The third line follows from the triangle inequality for the $L_{2p}$ norm
and Jensen's inequality. To bound the remaining
expectation, we simply note that the sum consists of independent,
random psd matrices.
We complete the proof by invoking the matrix Rosenthal
inequality \eqref
{eqn:rosenthal-pos}
and simplifying.

%
\section{\texorpdfstring{Proof of Theorem \lowercase{\protect\ref{cor:diag-perm-2}}}
{Proof of Theorem 10.3}}\label{sec:pf-diag-perm-2}

Consider the matrix Stein pair $(\mtx{X}, \mtx{X}')$ constructed in
Section~\ref{sec:diag-perm}. Expression~\eqref{eqn:diag-perm-DeltaX}
and the operator convexity \eqref{eqn:square-convex} of the matrix
square allow us to bound the conditional variance as follows.
\begin{eqnarray*}
\mtxx{\Delta}_{\mtx{X}}(\pi) &=& \frac{1}{4n} \sum
_{j,k=1}^n [ \mtx{A}_{j\pi(j)} +
\mtx{A}_{k \pi(k)} - \mtx{A}_{j \pi
(k)} - \mtx{A}_{k \pi(j)}
]^2
\\
&\psdle&\frac{1}{n} \sum_{j,k=1}^n
\bigl[ \mtx{A}_{j\pi(j)}^2 + \mtx{A}_{k \pi(k)}^2
+ \mtx{A}_{j \pi
(k)}^2 + \mtx{A}_{k \pi(j)}^2
\bigr]
\\
&=& 2 \sum_{j=1}^n \mtx{A}_{j \pi(j)}^2
+ \frac{2}{n} \sum_{j,k=1}^n
\mtx{A}_{j k}^2 = \mtx{W} + 4 \mtxx{\Sigma},
\end{eqnarray*}
where
\[
\mtx{W} \mathrell=2 \Biggl( \sum_{j=1}^n
\mtx{A}_{j\pi(j)}^2 \Biggr) - 2 \mtxx{\Sigma} \quad\mbox{and}\quad \mtx{
\Sigma} \mathrell=\frac{1}{n} \sum_{j,k=1}^n
\mtx {A}_{jk}^2.
\]
Substitute the bound for $\mtxx{\Delta}_{\mtx{X}}(\pi)$ into the
definition \eqref{eqn:r-psi} of $r(\psi)$ to
see that
%
\begin{eqnarray}
\label{eqn:r-perm-1} r(\psi) &\mathrell=&\frac{1}{\psi} \log\Expect\ntr
\econst^{\psi\mtxx
{\Delta}_{\mtx{X}}(\pi) }
\nonumber
\\[-8pt]
\\[-8pt]
\nonumber
&\leq&\frac{1}{\psi} \log\Expect\ntr\econst^{\psi(\mtx{W} + 4
\mtxx
{\Sigma}) } \leq4
\sigma^2 + \frac{1}{\psi} \log\Expect\ntr\econst^{\psi
\mtx{W}}.
\end{eqnarray}
The inequalities follow from the monotonicity of the trace
exponential \cite{Pet94Survey-Certain}, Section~2
and the fact that $\sigma^2 = \llVert {\mtxx{\Sigma}} \rrVert $. Therefore, it
suffices to bound the trace m.g.f. of~$\mtx{W}$.

Our approach is to construct a matrix Stein pair for $\mtx{W}$ and to
argue that the associated
conditional variance $\mtxx{\Delta}_{\mtx{W}}(\pi)$ satisfies a
semidefinite bound. We may then
exploit the trace m.g.f. bounds from Lemma~\ref{lem:mgf-bounds}.
Observe that $\mtx{W}$ and $\mtx{X}$ take the same form: both have mean
zero and share
the structure of a combinatorial sum. Therefore, we can study the
behavior of $\mtx{W}$ using the
matrix Stein pair from Section~\ref{sec:diag-perm}. Adapting \eqref
{eqn:diag-perm-DeltaX}, we
see that the conditional variance of $\mtx{W}$ satisfies
\begin{eqnarray*}
\mtxx{\Delta}_{\mtx{W}}(\pi) &=& \frac{1}{n} \sum
_{j,k=1}^n \bigl[ \mtx{A}_{j\pi(j)}^2
+ \mtx{A}_{k \pi(k)}^2 - \mtx{A}_{j \pi
(k)}^2 -
\mtx{A}_{k \pi(j)}^2 \bigr]^2
\\
&\psdle&\frac{4}{n} \sum_{j,k=1}^n
\bigl[ \mtx{A}_{j\pi(j)}^4 + \mtx{A}_{k \pi(k)}^4
+ \mtx{A}_{j \pi
(k)}^4 + \mtx{A}_{k \pi(j)}^4
\bigr]
\\
&\psdle&\frac{4R^2}{n} \sum_{j,k=1}^n
\bigl[ \mtx{A}_{j \pi(j)}^2 + \mtx{A}_{k \pi(k)}^2
+ \mtx{A}_{j
\pi
(k)}^2 + \mtx{A}_{k \pi(j)}^2
\bigr].
\end{eqnarray*}
In the first line, the centering terms in $\mtx{W}$ cancel each other
out. Then we apply the operator convexity \eqref{eqn:square-convex} of
the matrix square and the bound $\mtx{A}_{jk}^4 \psdle R^2 \mtx{A}_{jk}^2$.
Finally, identify $\mtx{W}$ and $\mtxx{\Sigma}$ to reach
%
\begin{equation}
\label{eqn:W-perm-bd} \mtxx{\Delta}_{\mtx{W}}(\pi) \psdle4R^2 (\mtx{W}
+ 4 \mtxx{\Sigma}) \psdle4R^2 \cdot\mtx{W} + 16R^2
\sigma^2 \cdot\Id.
\end{equation}
Matrix inequality \eqref{eqn:W-perm-bd} gives us access to established
trace m.g.f. bounds. Indeed,
\[
\log\Expect\ntr\econst^{\psi\mtx{W}} \leq\frac{8 R^2 \sigma^2 \psi^2}{1 - 4R^2 \psi}
\]
as a consequence of Lemma~\ref{lem:mgf-bounds} with parameters $c =
4R^2$ and $v = 16R^2 \sigma^2$.

At last, we substitute the latter bound into \eqref{eqn:r-perm-1} to
discover that
\[
r(\psi) \leq4\sigma^2 + \frac{8 R^2 \sigma^2 \psi}{1 - 4R^2 \psi}.
\]
In particular, setting $\psi= (8R^2)^{-1}$, we find that $r(\psi)
\leq
6\sigma^2$. Apply Theorem~\ref{thmm:concentration-subgauss} to wrap up.
\end{appendix}
\section*{Acknowledgments}
The authors thank Houman Owhadi for helpful conversations.
This paper is based on two independent manuscripts
from mid-2011 that both applied the method of exchangeable
pairs to establish matrix concentration inequalities.
One manuscript is by Mackey and Jordan; the other is
by Chen, Farrell and Tropp. The authors have
combined this research into a single unified presentation,
with equal contributions from both groups.

%

%


\printaddresses

\end{document}